\documentclass[11pt,review,hidelinks,onefignum,onetabnum,letterpaper]{article}
\usepackage[a4paper,margin=2.5cm]{geometry}
\usepackage[left,pagewise,mathlines]{lineno} % números en el margen exterior izquierdo del documento

\usepackage{amsmath,amsthm,amssymb}

\usepackage{lipsum}
\usepackage{amsfonts}
\usepackage{graphicx}
\usepackage{epstopdf}
\usepackage{algorithmic}
\usepackage{cleveref}
\ifpdf
  \DeclareGraphicsExtensions{.eps,.pdf,.png,.jpg}
\else
  \DeclareGraphicsExtensions{.eps}
\fi

\usepackage{enumitem}
\usepackage[justification=centering]{caption}
\DeclareGraphicsExtensions{.bmp,.png,.pdf,.jpg,.eps}

\usepackage{amsmath,amssymb,latexsym,amsfonts}
\usepackage{xcolor}
\definecolor{dAquaMarine}{rgb}{0, 0.5, 0.55}%Dark Aquamarine dolo

\usepackage[english]{babel}

\usepackage{multirow}
\usepackage{booktabs}

\newtheorem{theorem}{Theorem}
\newtheorem{lemma}{Lemma}

\newtheorem{problem}{Problem}
\newtheorem{remark}{Remark}
\newtheorem{corollary}{Corollary}
\def\nuc{\nu_{{}_{\mathrm C}\!}}

\newcommand*{\defeq}{\mathrel{\vcenter{\baselineskip0.5ex \lineskiplimit0pt \hbox{\scriptsize.}\hbox{\scriptsize.}}}=}

\def\At{A_{\theta}}
\def\Tt{T}

\def\Sht{\hat{S}_\theta}

\def\bV{\mathbf{V}}

\def\drdz{\,dr\,dz}
\def\dl{\,dl}

\def\L{\mathrm{L}}
\def\H{\mathrm{H}}

\def\O{\Omega}
\def\OC{\O_{{}_{\mathrm C}\!}}
\def\Ocero{\O_{{}_{0}\!}}
\def\Occ{\bar{\O}_{{}_{0}\!}}
\def\OA{\O_{{\mathrm a}\!}}
\def\Ok{\O_{{}_{k}\!}}

\def\IOcero{\mathrm{I}_{{}_{\Ocero}\!}}
\def\IOk{\mathrm{I}_{{}_{\Ok}\!}}
\def\Itt{\mathrm{I}_{{}_{\vartheta}\!}}

\def\Ot{\breve{\O}}

\def\Htu{\widetilde{\H}_1^1(\O)}
\def\LdGur{{\L}_{1}^2(\widehat{\Gamma}_1)}
\def\HtuG{\widetilde{\H}_{1,\Gamma}^1(\O)}
\def\HuT{\H_{1}^1(\Ocero)}
\def\Huc{\H_{1}^1(\Ocero)}

\def\HtuV{\widetilde{\H}_1^1(\vartheta)}

\def\HurV{\H_1^1(\vartheta)}
\def\HkrV{\H_1^k(\vartheta)}

\def\Ldr{\L_1^2(\O)}
\def\LdrV{\L_1^2(\vartheta)}
\def\LdurV{\L_{-1}^2(\vartheta)}
\def\Ldur{\L_{-1}^2(\O)}

\def\Ldrc{\L_1^2(\Ocero)}

\def\LdrOC{\L_1^2(\OC)}

\def\Lducoil{\L^2_{-1}(\O_{\mathrm{coil}})}

\def\scoil{\sigma_{\mathrm{coil}}}

\newcommand{\norm}[2]{\|#1\|_{#2}}

\def\R{\mathbb{R}}
\def\N{\mathbb{N}}
\def\C{\mathbb{C}}
\def\curl{\mathop{\mathbf{curl}}\nolimits}
\def\div{\mathop{\mathrm{div}}\nolimits}
\def\grad{\mathop{\mathbf{grad}}\nolimits}

\def\Ccal{\mathcal{C}}
\def\Acal{\mathcal{A}}
\def\X{\mathcal{X}}
\def\Y{\mathcal{Y}}
\def\D{\mathcal{D}}
\def\Th{\mathcal{T}_h}

\def\bE{\boldsymbol E}
\def\bB{\boldsymbol B}

\def\bH{\boldsymbol H}
\def\bJ{\boldsymbol J}
\def\bA{\boldsymbol A}
\def\bx{\boldsymbol x}
\def\bn{\boldsymbol n}
\def\be{\boldsymbol e}
\def\bp{\boldsymbol p}
\def\bq{\boldsymbol q}
\def\cero{\boldsymbol{0}}

\def\Gtu{\widehat{\Gamma}_1}
\def\Dt{\Delta t}
\def\dt{\,dt}
\def\ds{\,ds}
\def\Tf{\mathcal{T}}
\def\diff{\bar{\partial}}

\def\kcte{\mathring{\kappa}}

\def\eA{\texttt{e}_{h,A}}
\def\eT{\texttt{e}_{h,T}}

\def\dpt{\partial_t}

\usepackage{amsopn}
\newcommand{\keywords}[1]{%
\par\medskip\noindent\textbf{Keywords:} #1
}
\newcommand{\ccode}[1]{%
\par\medskip\noindent\textbf{AMS Subject Classification:} #1
}
\begin{document}
%\linenumbers

\title{\textbf{Numerical approximation of a transient thermo-electromagnetic problem in axisymmetric geometries}}

\author{
Dolores G\'omez\thanks{CITMAga, Departamento de Matem\'atica Aplicada, Universidade de Santiago de Compostela, Spain. \texttt{mdolores.gomez@usc.es}}
\and
Bibiana L\'opez-Rodr\'iguez\thanks{Departamento de Matem\'aticas, Universidad Nacional de Colombia, sede Medell\'in, Colombia. \texttt{blopezr@unal.edu.co}}
\and
Pilar Salgado\thanks{CITMAga, Departamento de Matem\'atica Aplicada, Universidade de Santiago de Compostela, Spain. \texttt{mpilar.salgado@usc.es}}
\and
Pablo Venegas\thanks{GIMNAP, Departamento de Matem\'atica, Universidad del B{\'\i}o-B{\'\i}o, Concepci\'on, Chile. \texttt{pvenegas@ubiobio.cl}}
}

\author{
Dolores G\'omez\thanks{CITMAga, Departamento de Matem\'atica Aplicada, Universidade de Santiago de Compostela, Santiago de Compostela, Spain.
Email: \texttt{mdolores.gomez@usc.es}.}
\and
Bibiana L\'opez-Rodr\'iguez\thanks{Departamento de Matem\'aticas, Universidad Nacional de Colombia, sede Medell\'in, Colombia.
Email: \texttt{blopezr@unal.edu.co}.}
\and
Pilar Salgado\thanks{CITMAga, Departamento de Matem\'atica Aplicada, Universidade de Santiago de Compostela, Santiago de Compostela, Spain.
Email: \texttt{mpilar.salgado@usc.es}.}
\and
Pablo Venegas\thanks{GIMNAP, Departamento de Matem\'atica, Universidad del B{\'\i}o-B{\'\i}o, Concepci\'on, Chile.
Email: \texttt{pvenegas@ubiobio.cl}.}
}
%----------------------------------------------------------------------

\date{}

%----------------------------------------------------------------------
\maketitle
%----------------------------------------------------------------------

\begin{abstract}
This paper analyzes a transient thermo-electromagnetic problem arising in the modeling of induction heating processes. Unlike previous studies that focused on steady-state scenarios, we consider a time-dependent thermal problem coupled with a nonlinear time-harmonic electromagnetic problem through  temperature-dependent electrical conductivity and Joule effect. Exploiting cylindrical symmetry and assuming a purely azimuthal current density, we formulate the problem on a two-dimensional meridional section. We introduce a variational formulation in appropriately weighted Sobolev spaces and prove existence of a solution  by a fixed-point argument. Under reasonable assumptions on the physical parameters, we also prove uniqueness. A finite element discretization combined with implicit time stepping is used to compute the numerical solution. To evaluate the accuracy of the approximation, a priori error estimates are derived and validated by numerical experiments.  Finally, numerical simulations illustrate the effectiveness of the proposed approach in an industrially relevant configuration.

%----------------------------------------------------------------------


\medskip
\ccode{65N30, 65M60, 35K61}

\end{abstract}

%-------------------------------------------------------------------------------------
\section{Introduction}
\label{sec:intro}
%-------------------------------------------------------------------------------------
Thermo-electromagnetic models describe the coupled evolution of temperature and
electromagnetic fields and, in most applications, take the form of nonlinear systems of
partial differential equations.  This work addresses the numerical analysis of coupled
thermo-electromagnetic systems arising in induction heating, a technique in which an
alternating current in an induction coil generates a time-varying magnetic field that
induces eddy currents in an electrically conductive workpiece. The resulting resistive
losses (Joule heating) raise the workpiece temperature, which can be controlled through the
excitation frequency and current as well as the coil geometry. Due to their flexibility and
reliability, induction heating systems have become an indispensable part of a wide range of
industrial applications, including material forming and induction hardening. From a
modeling viewpoint, the process requires solving an eddy-current problem to compute the
dissipated power together with a heat-transfer problem for the temperature
field. The coupling is twofold: the electrical conductivity depends on the temperature and
Joule heating provides the source term in the thermal equation. The overall problem is
nonlinear and becomes particularly challenging for ferromagnetic materials, where the
magnetic constitutive relation between the magnetic field and the magnetic induction is
nonlinear.

In a previous study~\cite{GLSV1} we analyzed the steady-state
ther\-mo-elec\-tro\-mag\-net\-ic model
as a first step toward  more realistic simulations of induction heating. Here, we extend
that work by coupling a transient thermal problem with a nonlinear time-harmonic
electromagnetic formulation to capture the temperature evolution.
While the harmonic formulation is exact for linear materials under harmonic excitation,
in the present nonlinear setting, it must be considered an approximation of the
electromagnetic response. This is a justified choice, as the large separation between
the electromagnetic and thermal time scales renders a fully transient nonlinear
coupled simulation computationally prohibitive.  For this reason, nonlinear
time-harmonic formulations are widely used in commercial software for industrial
induction heating  simulations.

Compared with the steady-state case, the proposed transient analysis
is more delicate. The main difficulties arise from the additional nonlinear magnetic
term associated with the B–H curve and from the Joule heating source, which
is computed from the electromagnetic solution and may have low regularity (often only in $\L^1$
in space). This complicates the analysis of both the heat equation and the coupled
problem. In particular, time regularity is needed to apply compactness
arguments and establish the existence of a weak solution. Moreover, the time dependence of the
heat equation introduces further challenges, including the need to propose, implement,
and analyze the convergence of an efficient and reliable time discretization scheme. Existence
results for fully transient thermo-electromagnetic models have been analyzed using
time discretization and compactness techniques; see, for example,
\cite{CGS2017,CS2017,WYZ2020}, where Rothe's method is used.

We focus on a two-dimensional axisymmetric setting. Both the underlying mathematical
model and the numerical strategy are standard in the induction heating
literature and have been extensively validated (see, e.g.,
\cite{CCGMRS97,BGMS2007,BGMSV2009}). %In addition, the \cmag{linear} electromagnetic
%subproblem has been studied from both mathematical and numerical viewpoints (see \cite{BRRSIMA}).
However, to the best of our knowledge, neither the well-posedness of the coupled weak formulation, including the nonlinear magnetic effects associated with the
B-H curve, nor a rigorous numerical analysis of a fully discrete approximation has been established so far. This is the main goal of the present work.

%However, to the best of our knowledge, the well-posedness of the coupled weak formulation, as well as a rigorous numerical analysis of a finite element approximation for it, has not yet been established; this is the main objective of the present work.

The resulting thermo-electromagnetic model is structurally similar to the classical
thermistor problem, for which substantial theoretical and numerical literature is available
(\cite{Allegretto1992,Akrivis2005,LGS14,GLS2017,Mbehou2018,SH2018,GSW2021,JMP2022}).
However, the axisymmetric geometry and the specific features of the electromagnetic
subsystem considered here lead to a coupled partial differential equation system that
differs substantially from the standard thermistor setting, and therefore requires a
dedicated variational framework and numerical analysis.

For the electromagnetic problem, we adopt a nonlinear time-harmonic eddy-current
formulation written in terms of the magnetic vector potential. This formulation is
coupled to a transient nonlinear heat equation on the workpiece, with Joule heating as
the source term. Under appropriate assumptions, the three-dimensional problem reduces to
a two-dimensional model on a meridional section. In this setting, we introduce weighted
Sobolev spaces and derive a coupled weak formulation in which the magnetic vector potential
is posed on a computational domain containing the workpiece, the coils and the surrounding
air, whereas the temperature is defined only on the workpiece.

To establish the well-posedness of the coupled thermo-electromagnetic problem, we use a
fixed-point argument similar to that in \cite{BMP2005,BMP2013}, where existence is proved
for three-dimensional transient thermo-electromagnetic models arising in the simulation
of electrodes in electric furnaces. Related existence results for fully transient
eddy-current formulations written in terms of the magnetic vector potential can also
be found in \cite{CGS2017}. Our setting, however, requires a different functional framework:
the axisymmetric reduction leads naturally to cylindrical coordinates and weighted
Sobolev spaces, and the associated degeneracies near the symmetry axis introduce
additional analytical difficulties, even when dealing with a transient linear thermal
model. Under further assumptions on the material coefficients (particularly, on the
electrical and the thermal conductivities), we also prove uniqueness of the coupled
solution. In this work, our contribution goes beyond the continuous analysis. We provide
a rigorous foundation for a fully discrete scheme based on backward Euler time
stepping and a finite element discretization in space. Assuming additional regularity
of the exact solution, we derive optimal-order a priori error estimates for the
resulting approximation.

The paper is organized as follows: In ~\cref{sec:stat}, we formulate the transient
thermo-electromagnetic problem in an axisymmetric setting, using a magnetic vector
potential for the electromagnetic component and prescribing voltages as input data. In~\cref{sec:ex_uni} we introduce the weak formulation, prove well-posedness via Schauder’s fixed-point theorem, and establish uniqueness under additional assumptions on the material coefficients. In~\cref{sec:fully} we present and analyze a fully discrete scheme based on backward Euler time stepping and Lagrange finite elements in space; we prove convergence and derive optimal-order a priori error estimates under additional regularity assumptions. Finally, in~\cref{sec:num} we report numerical tests that confirm the theoretical convergence rates and illustrate the method on an industrially relevant induction-hardening configuration from the automotive industry.

\section{Statement of the problem}
\label{sec:stat}
%-------------------------------------------------------------------------------------

\begin{figure}[!ht]
\centering
\includegraphics*[scale=0.45]{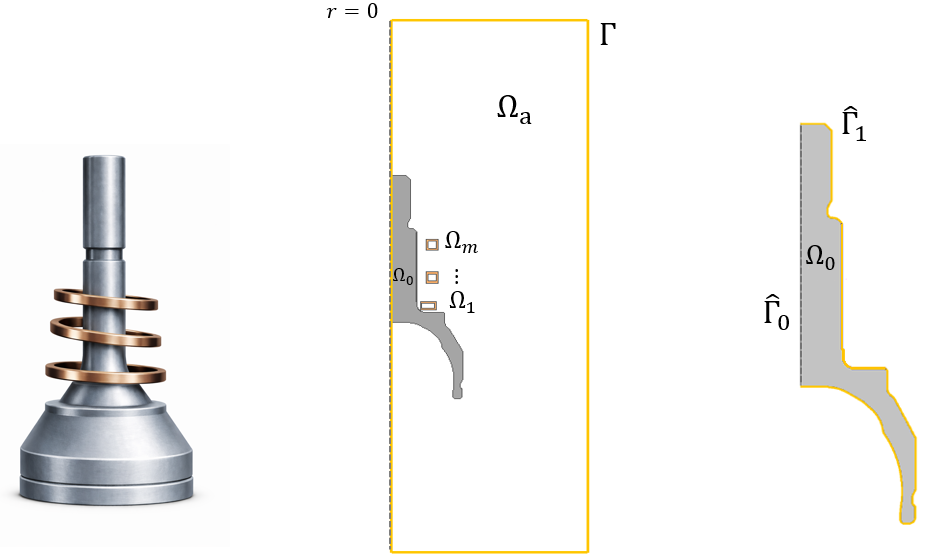}
\caption{Induction heating system including the workpiece and the coil (left). Meridional
domains for the electromagnetic (center) and thermal (right) problems.}
\label{fig:domainaxi_notations}
\end{figure}

We consider a basic induction-heating system consisting of a coil surrounding an axially
symmetric conductive workpiece, as shown in Figure~\ref{fig:domainaxi_notations}. For the axisymmetric
formulation, the coil is replaced by $m$ parallel toroidal rings. The detailed derivation of the model and
the notation follow \cite{GLSV1}, except that the thermal equation is now transient and the electromagnetic
model may involve a nonlinear B-H curve. In particular, the three-dimensional computational domain is a
cylinder $\Ot$ enclosing the heating system and chosen ‘sufficiently large’, whose meridian section $\Omega$
is decomposed as $\Omega=\OC\cup \OA$. The conductor region is $\OC=\Omega_0\cup\Omega_{\rm coil}$, where
$\Omega_0$ is the workpiece section, assumed to be polygonal but not necessarily convex, and
$\Omega_{\rm coil}=\bigcup_{k=1}^m \Omega_k$ is the union of the $m$  coil ring sections. The
surrounding air is denoted by $\OA$.  Finally, we split $\Gamma:=\partial\Omega=\Gamma_0\cup\Gamma_1$
and $\widehat{\Gamma}:=\partial\Omega_0=\widehat{\Gamma}_0\cup\widehat{\Gamma}_1$, where $\Gamma_0$
and $\widehat{\Gamma}_0$ are the parts lying on the symmetry axis $r=0$.

The electromagnetic excitation is time-harmonic with angular frequency $\omega=2\pi f$ and we adopt the eddy current approximation, as in \cite{BGSbook,GLSV1}. It reads
\begin{equation*}%\label{eq:maxwell}
i\omega\bB+\curl\bE =\cero,\quad \curl\bH =\bJ,\quad \div\bB =0,\qquad \text{in } \Ot,
\end{equation*}
where $\bE$, $\bB$, $\bH$ and $\bJ$ are the complex amplitudes of the electric field, magnetic induction,  magnetic field and  current density, respectively.

For the axisymmetric formulation, we use cylindrical coordinates $(r,\theta,z)$ with associated local orthonormal basis $(\be_r,\be_{\theta},\be_z)$, the symmetry axis being $r=0$. We assume that all fields and coefficients are independent of $\theta$, and that the current density in conductors has the form $\bJ(r,\theta,z)=J_\theta(r,z)\be_{\theta}$. Then, the magnetic vector potential has only an azimuthal
component,
\[
\bA(r,\theta,z)=A_\theta(r,z)\,\be_\theta,
\qquad \bB=\curl \bA.
\]
To account for magnetic nonlinearity, we consider the constitutive law
\begin{equation} \label{HnuB}
\bH=\nu(|\bB|)\bB
:=
\begin{cases}
\nuc(|\bB|) \bB, & \text{in } \OC,\\
\nu_0 \bB, & \text{in } \OA,
\end{cases}
\end{equation}
where $\nuc$ denotes the magnetic reluctivity in the conducting region
and $\nu_0>0$ the vacuum reluctivity. In $\OC$,  Ohm's law
reads $\bJ=\sigma(T)\bE$, with temperature-dependent electrical conductivity $\sigma$, which couples the electromagnetic and thermal models.

Following \cite{GLSV1}, there exist complex constants $V_k$,
$k=0,\dots,m$ such that
\begin{equation}\label{eq:JVA}
i\omega A_\theta+\sigma(T)^{-1}J_\theta=\frac{V_k}{r}\qquad \text{in } \Omega_k.
\end{equation}
Moreover, in the workpiece, $V_0=0$. Hence, the electromagnetic model reads:
\begin{align}\label{EM}
-\left(
\frac{\partial}{\partial r}
\left(\frac{\nu}{r}\frac{\partial(r\At)}{\partial r}\right)
+\frac{\partial}{\partial z}
\left(\nu\frac{\partial(\At)}{\partial z}\right)
\right)
=\begin{cases}
-i\omega \sigma \At
+ \sigma \dfrac{V_k}{r} &\text{in }\Ok,\  k=1,\dots,m,\\
\noalign{\smallskip}
-i\omega \sigma \At &\text{in }\Ocero,\\
\noalign{\smallskip}
\phantom{-}0&\text{in }\OA,
\end{cases}
\end{align}
where $\nu=\nu(|\curl(A_\theta\be_{\theta})|)$. We impose the boundary condition  $\At=0$ on $\Gamma$.
% \begin{equation}\label{BCE}
% \At=0  \quad\text{on }\Gamma,
% \end{equation}

The transient thermal is posed only in the workpiece. Under the axisymmetry assumption, the temperature is written as $T=T(t,r,z)$, and for all $t \in [0,\Tf]$  satisfies
\begin{equation}\label{Ttheta}
\rho c_P\frac{\partial T}{\partial t}
-\frac{1}{r}\frac{\partial}{\partial r}\left(r\kappa(T)\frac{\partial T}{\partial r}\right)
-\frac{\partial}{\partial z}\left(\kappa(T)\frac{\partial T}{\partial z}\right)
= \frac{\omega^2}{2}\sigma(T)\,|A_\theta|^2 \,\,\text{ in } \Omega_0.
\end{equation}
Here, $\rho$, $c_P$ and $\kappa$ denote the density, specific heat, and temperature-dependent thermal conductivity, respectively. The right-hand side represents the cycle-averaged Joule heating $|\bJ|^2/(2\sigma)$, obtained from \eqref{eq:JVA} with $V_0=0$, and provides the coupling with the electromagnetic model. Equation \eqref{Ttheta} is complemented with boundary conditions
\begin{equation}\label{BCT}
\begin{aligned}
\kappa(T)\grad T\cdot \bn + \eta\,(T-T_c)&=0\qquad \text{on }\widehat{\Gamma}_1,\\
\kappa(T)\grad T\cdot \bn &=0\qquad \text{on }\widehat{\Gamma}_0,
\end{aligned}
\end{equation}
and with the initial condition $T(0,\cdot)=T_0(\cdot)$ in $\Omega_0$.

%-------------------------------------------------------------------------------------
\section{Mathematical analysis}\label{sec:ex_uni}

% In this section, we make a precise statement of the problem to be solved by means
% of a weak formulation suitable for its mathematical analysis. Then, we prove the
% existence and, under additional assumptions, the uniqueness of a solution. First,
% we introduce some preliminary results which will be used along the paper.
In this section, we state the problem in weak form, in a way suitable for the mathematical analysis. We then prove existence and, under additional assumptions, uniqueness of the solution. To this end, we first introduce some preliminary results and notation. Derivatives are denoted in the standard way with respect to the variable under consideration. In particular, the notations  $\partial_u g$ and $\frac{\partial g}{\partial u}$
are used interchangeably for partial derivatives with respect to $u$.

%-------------------------------------------------------------------------------------
\subsection{Functional spaces and preliminary results}

We define weighted Sobolev spaces appropriate for the mathematical analysis of
the problem and recall some of their properties. For $\vartheta \subseteq \O$ and $\alpha \in \R$, let $\L^2_{\alpha}(\vartheta)$ be the weighted
Lebesgue space of measurable functions $G$ defined in $\vartheta$ with bounded norm
\[
\norm{G}{\L^2_{\alpha}(\vartheta)}^2:=\int_{\vartheta} |G|^2 r^{\alpha} \drdz.
\]
Given $k\in\N$, the weighted Sobolev space $\HkrV$ consists of all functions in $\LdrV$ whose derivatives in the sense
of distributions up to order $k$ are also in $\LdrV$. We define the norms and semi-norms of these spaces in the standard
way.
%; for instance, in $\HurV$ the semi-norm
%\[
%\left|G\right|^2_{\HurV} :=\int_{\vartheta}\left(\left|\pdr G\right|^2
%+\left|\pdz G\right|^2\right)r\drdz.
%\]
We will also use the following Hilbert space:
\[
\HtuV:=\HurV\cap\LdurV,
%\HtdV&:=\left\{Z\in\HtuV:\;\norm{Z}{\HtdV}<\infty\right\},
\]
with its respective norm defined by
\[
\norm{Z}{\HtuV}^2 :=\norm{Z}{\HurV}^2+\norm{Z}{\LdurV}^2.
%\norm{Z}{\HtdV}^2 &:=\norm{Z}{\HtuV}^2
%+\left|r^{-1}\pdr(rZ)\right|_{\HurV}^2 +\norm{\pdz Z}{\HtuV}^2.
\]

In addition, for our analysis, we will rely on the results presented below whose proofs
can be found in references \cite{BM1999}, \cite{BRRSIMA} and \cite{GP2006}, respectively.
\begin{lemma}
\label{lemma:emb_cont}
Let $\vartheta \in \{\O,\Ocero\}$, there holds $\H^1_1(\vartheta) \subset\L^6_1(\vartheta)$ continuously, and there
exists $C_e>0$ such that $\norm{G}{\L^6_1(\vartheta)}\leq C_e \norm{G}{\H^1_1(\vartheta)}$ for all
$G\in\H^1_1(\vartheta)$.
\end{lemma}
\begin{lemma}\label{lemma:Poincare}
For all $W \in \HuT$, there exists $C_p>0$ such that
\[
C_p\,\norm{W}{\HuT}^2\leq |W|_{\HuT}^2 +\norm{W}{\LdGur}^2.
\]
\end{lemma}
\begin{lemma}\label{lemma:Traza}
For all $W \in \HuT$, there exists $C_{\mathrm{tr}}>0$ such that
\[
\norm{W}{\LdGur}\leq C_{\mathrm{tr}}\norm{W}{\HuT}.
\]
\end{lemma}
\begin{lemma}
\label{lemma:Htu}
For all $Z \in \Htu$, $\partial_r(rZ)$ belongs to $\Ldur$ and satisfies
\[
\norm{\partial_rZ}{\Ldr}^2+ \norm{Z}{\Ldur}^2
\leq\norm{\partial_r(rZ)}{\Ldur}^2\leq 2\norm{\partial_rZ}{\Ldr}^2+2\norm{Z}{\Ldur}^2.
\]
\end{lemma}

Moreover, we recall the classical
functional framework for functions defined on a bounded interval $[0,\Tf]$
 with values in a separable Hilbert space $X$. We denote by
$\Ccal([0,\Tf];X)$  the Banach space  of all continuous
functions $g:[0,\Tf]\to X$. We also consider the space $\L^2(0,\Tf;X)$
% of
% classes of functions $g:[0,\Tf]\to X$ that are B\"{o}chner-measurable and
% such that
of equivalence classes of  B\"{o}chner-measurable functions
$g:[0,\Tf]\to X$ such that
$\|g\|_{\L^2(0,\Tf;X)}
:=\left(\int_0^{\Tf}\|g(t)\|_{X}^2\,dt\right)^{1/2}<\infty$. Furthermore,
we define $$\H^1(0,\Tf;X):=\left\{g\in\L^2(0,\Tf;X):
\ \dpt g\in\L^2(0,\Tf;X)\right\}.$$
% (we will use indistinctly the
% notations $\dpt g$ and $dg/dt$ for the derivative
% with respect to the variable $t$).
Analogously, $\H^k(0,\Tf;X)$ is defined
for every $k\in\N$.

In what follows, let us suppose that $\sigma(\bx,u)$, $\kappa(\bx,u)$ are Carath\'eodory functions from
$\O\times\R$ into $\R$, i.e., measurable with respect to $\bx:=(r,\theta)$ and continuous with respect to
$u$. Moreover, we make the following assumptions on the data
\begin{align}
&\eta_*\leq\eta(\bx)\leq\eta^*\ \text{ a.e. }\bx\in\Gtu,\label{hyp:eta}\\
&\kappa_*\leq\kappa(\bx,u)\leq\kappa^*\ \text{ a.e. }\bx\in\Ocero,\ \forall u\in \R,\\
&\sigma_*\leq\sigma(\bx,u)\leq\sigma^*\ \text{ a.e. }\bx\in\Ocero, \ \forall u\in \R,\label{hyp:sigma}
\end{align}
where $ \eta_*^{(*)},\kappa_*^{(*)}$ and $\sigma_*^{(*)}$   are strictly positive constants.
% In addition, the electrical conductivity in the coil is assumed to be constant
% \begin{equation}
% \label{hyp:sigmak}
% \sigma(\bx,u)= \sigma_k>0\ \text{ a.e. }\bx\in\Ok,\ k=1,\dots,m,
% \text{ and }\sigma(\bx,u)= 0\ \text{ a.e. }\bx\in\OA\,,\ \forall u\in \R.
% \end{equation}
In addition, $\sigma$ is assumed to be constant in each coil ring and equal to zero in the air region, namely,
 \begin{equation}
\label{hyp:sigmak}
\sigma(\bx,u)=
\begin{cases}
\sigma_k>0, & \text{a.e. in }\Omega_k,\quad k=1,\dots,m,\\[1mm]
0, & \text{a.e. in }\OA,
\end{cases}
\qquad \forall u\in\mathbb R.
\end{equation}
We denote by $\scoil\defeq \underset{1\leq k\leq m}{\max}\,\sigma_k$, the largest electrical conductivities among
the coil rings and for any $\vartheta \subseteq \O$, we define $\Itt\defeq \int_{\vartheta}r\drdz$. Let us also consider the
following assumptions on the nonlinear magnetic reluctivity  $\nuc:\mathbb{R}_0^+\to \mathbb{R}$ (see, e.g., \cite{Pechstein})
\begin{align}
\exists\, \nu_{\textrm{min}}\,&:\,\, 0<\nu_{\textrm{min}}\leq \nuc(p)\leq\nu_0&& \forall p\in \R_0^+\,, \label{nuA}\\
\exists\, \xi>0\,&:\,|\nuc(p)p-\nuc(q)q|\leq \xi|p-q|&&\forall p,q \in \R_0^+\,, \label{nuL}\\
\exists\, \varrho>0\,&:\, (\nuc(p)p-\nuc(q)q)(p-q)\geq\varrho|p-q|^2&& \forall p,q \in \R_0^+\,.\label{nuSM}
\end{align}
The previous properties are natural assumptions on the reluctivity consistent with physical $BH$-  curves.
 Finally, for simplicity, we  assume that $\rho= c_P = 1$.

 In order to establish a weak formulation of problem \eqref{EM}--\eqref{BCT}, we introduce
 %we consider the following subspace of
% $\Htu$:
%
\[
\HtuG:=\left\{Z\in\Htu: \,\, Z=0 \,\, \mbox{ on } \Gamma\right\}.
\]
We integrate \eqref{EM} multiplied by a test function $rZ$ with $Z\in\HtuG$ and \eqref{Ttheta} by $rW$ with $W\in\HuT$,
integrate by parts and use \eqref{BCT}. Thus, we are led to the following nonlinear coupled variational problem:
\begin{problem}\label{pbm:AT_time0}
Given $\bV\defeq(V_1,\dots,V_m)\in\C^m$, $T_c\in\L^2(0,\Tf;\LdGur)$ and $T_0\in\Ldrc$, find
$(\At,\Tt)\in\L^2(0,\Tf;\HtuG)\times \L^2(0,\Tf;\HuT)\cap\H^1(0,\Tf;\Ldrc)$ such that for all $(Z,W)\in\HtuG\times\HuT$
and a.e. $t\in (0,\Tf)$
\begin{align*}
i\omega\int_{\OC}\sigma(\cdot,\Tt) \At\bar{Z}r\drdz
+\int_\O \nu(|\curl(\At\be_{\theta})|)
\left(
\frac{1}{r}\frac{\partial(r\At)}{\partial r}\frac{1}{r}\frac{\partial(r\bar{Z})}{\partial r}
+\frac{\partial \At}{\partial z}\frac{\partial \bar{Z}}{\partial z}
\right)r\drdz\\
% i\omega\int_{\OC}\sigma(\cdot,\Tt) \At\bar{Z}r\drdz
% +\int_\O \nu(|\curl(\At\be_{\theta})|)
% \left(
% \frac{1}{r^2}\partial_r(r\At)\partial_r(r\bar{Z})
% +\partial_z \At\partial_z \bar{Z}
% \right)r\drdz\\
 =\sum_{k=1}^m\int_{\Ok}\sigma_k V_k\bar{Z}\drdz,\\
\int_{\Ocero}\partial_t\Tt Wr\drdz+
\int_{\Ocero}\kappa(\cdot,\Tt)\grad \Tt\cdot\grad Wr\drdz+\int_{\Gtu}\eta\, \Tt\,Wr\dl\\
=\frac12\int_{\Ocero}\omega^2\sigma(\cdot,\Tt)|\At|^2Wr\drdz
+\int_{\Gtu}\eta\,T_c\,Wr\dl ,\\
\Tt(0)=T_0\qquad \mbox{in }\Ocero.
\end{align*}
\end{problem}
In the upcoming section, we will establish the existence of solution to the
preceding problem through the application of a fixed-point argument.
 Additionally, under additional assumptions on the coefficients, we will
provide a proof for the uniqueness of the solution.

%-------------------------------------------------------------------------------------
%\section{Existence and uniqueness}
%-------------------------------------------------------------------------------------
\subsection{Existence}

To establish the existence of a solution to \cref{pbm:AT_time0},
we employ Schauder's fixed-point theorem. To this end, we define two operators: one corresponding to the nonlinear electric problem and another for the nonlinear thermal problem. Both operators are shown to be bounded and continuous.
%
%To achieve this, we define two operators corresponding to the nonlinear electric and thermal problems,
%respectively, that are bounded and continuous.
%To prove the existence of  solution to \cref{pbm:AT_time0}, let us first show a useful
%result to prove that the nonlinear electric problem has a unique solution.
%On the other hand, using the classical theory of parabolic problems,
%we guarantee a unique solution to thermal problem. This allows us
%to define two bounded and continuous operators that, via Schauder's
%fixed-point theorem, we can proof the existence of a solution to
%\cref{pbm:AT_time0}.

Let us first introduce the nonlinear operator $\Acal:\Htu\to\Htu'$, that will be used to study the electrical problem:
\begin{equation}\label{eq:Acal}
\langle \Acal(G),Z\rangle \defeq \int_\O \nu(|\curl(G\be_{\theta})|)
\left(
\frac{1}{r}\frac{\partial(rG)}{\partial r}\frac{1}{r}\frac{\partial(r\bar{Z})}{\partial r}
+\frac{\partial G}{\partial z}\frac{\partial \bar{Z}}{\partial z}
\right)r\drdz\quad\forall G,Z\in\Htu.
\end{equation}
Notice that, using the expression of the {curl} operator in cylindrical coordinates,  $\Acal$ can be rewritten as %follows (cf.\eqref{curlF})
\[
\langle\Acal(G),Z\rangle
= \int_\O \nu(|\curl(G\be_{\theta})|) \curl(G\be_{\theta})\cdot\curl(\bar{Z}\be_{\theta})\,r\drdz.
\]
Next we show several key  properties of this operator which will be instrumental in various contexts throughout this article.

\begin{lemma}\label{lemma:Acal}
The nonlinear operator $\Acal$ is strongly monotone and Lipschitz continuous, i.e., there exist $\tilde{\xi}$ and $\tilde{\varrho}$ strictly positive constants such that
\begin{align}
\label{AcalSM}
\langle\Acal(G)-\Acal(Z),G-Z\rangle&\geq \tilde{\varrho}\norm{G-Z}{\Htu}^2&&\forall G,Z\in\Htu,\\
\label{AcalL}
|\langle\Acal(G)-\Acal(F),Z\rangle|&\leq \tilde{\xi} \norm{G-F}{\Htu}\,\norm{Z}{\Htu}&&\forall G,F,Z\in\Htu,
\end{align}
where $\tilde{\varrho}\defeq \frac{\min\{\varrho,\nu_0\}}{2\max\{1, R^2\}}$,  $\tilde{\xi}\defeq 4\max\{3\xi,\nu_0\}$
and  $R$ denoting the width of the rectangle $\O$.
\end{lemma}
\begin{proof}
We first notice that, from \eqref{nuA}--\eqref{nuSM}  and by proceeding as in the proof of \cite[Lemma~2.8 and Lemma~2.9]{Pechstein}
it follows that $\nuc :\C^3\to\C^3$ is strongly monotone and Lipschitz continuous; in fact
\begin{equation}
\label{eq:mon_lc}
(\nuc(|\bp|)\bp-\nuc(|\bq|)\bq)\cdot(\bar{\bp}-\bar{\bq})\geq\varrho|\bp-\bq|^2,\quad
|\nuc(|\bp|)\bp-\nuc(|\bq|)\bq|\leq 3\xi|\bp-\bq|
\qquad \forall \bp,\bq \in \C^3.
\end{equation}
Thus, from the definition of $\nuc$ (cf.~\eqref{HnuB}), $\Acal$ and by taking $\bp=\curl(G\be_{\theta})$ and $\bq=\curl(Z\be_{\theta})$ in the previous equation we obtain
\begin{align*}
&\langle\Acal(G)-\Acal(Z),G-Z\rangle\\
&=\int_\O \left(\nu(|\curl(G\be_{\theta})|)\curl(G\be_{\theta})-\nu(|\curl(Z\be_{\theta})|) \curl(Z\be_{\theta})\right)\cdot\curl(\bar{G}\be_{\theta}-\bar{Z}\be_{\theta})r\drdz\\
&\geq \varrho\int_{\OC} |\curl(G\be_{\theta}-Z\be_{\theta})|^2r\drdz
+\nu_0\int_{\OA} |\curl(G\be_{\theta}-Z\be_{\theta})|^2r\drdz\\
&\geq \min\{\varrho,\nu_0\}\norm{\curl(G-Z)\be_{\theta}}{\Ldr}^2
=\min\{\varrho,\nu_0\}\left(\norm{\partial_r(r(G-Z))}{\Ldur}^2+\norm{\partial_z (G-Z)}{\Ldr}^2\right)\\
&\geq\frac{\min\{\varrho,\nu_0\}}{2\max\{1, R^2\}}\norm{G-Z}{\Htu}^2
\end{align*}
where, for the last inequality, we have applied \cref{lemma:Htu}.
Next, we apply \eqref{eq:mon_lc}  and \cref{lemma:Htu} again to prove that $\Acal$ is Lipschitz continuous
\begin{align*}
&|\langle\Acal(G)\!-\!\Acal(F),Z\rangle|
\leq \!\int_\O \!\!\left|\nu(|\curl(G\be_{\theta})|)\curl(G\be_{\theta})\!-\!\nu(|\curl(F\be_{\theta})|) \curl(F\be_{\theta})\right|\,|\curl(\bar{Z}\be_{\theta})|r\drdz\\
&\leq 3\xi\int_{\OC} |\curl(G\be_{\theta}-F\be_{\theta})|\,|\curl(\bar{Z}\be_{\theta})|r\drdz
+\nu_0\int_{\OA} |\curl(G\be_{\theta}-F\be_{\theta})|\,|\curl(\bar{Z}\be_{\theta})|r\drdz\\
&\leq \max\{3\xi,\nu_0\}\norm{\curl(G-F)\be_{\theta}}{\Ldr}\norm{\curl(Z)\be_{\theta}}{\Ldr}\\
&\leq 4\max\{3\xi,\nu_0\}\norm{G-F}{\Htu}\norm{Z}{\Htu}.
\end{align*}
Thus, the proof is complete.
\end{proof}

%As a consequence of the previous lemma it follows that
%\begin{equation}
%\label{AcalC}
%\langle\Acal(Z),Z\rangle\geq \tilde{\varrho} \norm{Z}{\Htu}^2\qquad \forall Z\in\Htu.
%\end{equation}
%-------------------------------------------------------------------------------------

Next, in order to study the existence of solution of the coupled problem, we begin by introducing the nonlinear
operator related to the electromagnetic problem: given
$S\in\L^2(0,\Tf;\Ldrc)$, we defined $\mathcal{H}$ as follows
\[
\begin{array}{ccccc}
&\mathcal{H}:&\L^2(0,\Tf;\Ldrc)&\longrightarrow&\L^2(0,\Tf;\HtuG)\\
& & S&\longmapsto&\mathcal{H}(S)\defeq \At^S
\end{array}
\]
where, a.e. $t\in (0,\Tf)$, $\At^S(t)$ solves, $\forall Z\in\HtuG$,
\begin{equation}\label{eq:pb_H}
i\omega\int_{\OC}\!\!\!\sigma(\cdot,S) \At^S(t)\bar{Z}r\drdz
+\langle\Acal(\At^S(t)),Z\rangle
=\sum_{k=1}^m\int_{\Ok}\!\!\! \sigma_k V_k\bar{Z}\drdz%\quad\forall Z\in\HtuG.
\end{equation}
The following lemma guarantees the well-posedness of the preceding problem.
\begin{lemma}\label{lemma:H_eub}
Let $\bV\in\C^m$. For $S\in\L^2(0,\Tf;\Ldrc)$ there exists a unique
$\At^S \in \L^2(0,\Tf;\HtuG)$ solution  of problem~\eqref{eq:pb_H}. Moreover
\begin{align*}
\norm{\At^S}{\L^2(0,\Tf;\Htu)}&\leq
C_c\norm{\bV}{\infty}
\end{align*}
where $C_c\defeq \scoil\, \tilde{\varrho}^{-1}\left(\sum_{k=1}^m \IOk \right)^{\frac12}\sqrt{\Tf}$ and $\norm{\bV}{\infty}\defeq \underset{1\leq k\leq m}{\max}\{|V_k|\}$.
\end{lemma}
\begin{proof}
The existence result follows from \cite[Theorem 25.B]{zeidler2b} since the operator $\Acal$ is strongly monotone and Lipschitz continuous. In fact both, the strongly monotone and Lipschitz continuous  follows from \cref{lemma:Acal}.
To prove the estimate for $\At^S$, let us first notice that, as a consequence of \cref{lemma:Acal}
 we get
\[
\langle\Acal(Z),Z\rangle\geq \tilde{\varrho} \norm{Z}{\Htu}^2\qquad \forall Z\in\Htu.
\]
Thus, by taking $Z=\At^S$ in \eqref{eq:pb_H}, using the previous inequality and proceeding as in the  proof of \cite[Lemma~3.5]{GLSV1} it follows
\begin{align*}
\tilde{\varrho}\norm{\At^S}{\Htu}^2
&\leq \mathrm{Re}\left(i\omega\int_{\OC}\sigma(\cdot,S) |\At^S|^2r\drdz
+\langle\Acal(\At^S),\At^S\rangle\right)\\
&\leq \left|i\omega\int_{\OC}\sigma(\cdot,S) |\At^S|^2r\drdz
+\langle\Acal(\At^S),\At^S\rangle\right|
=\left|\sum_{k=1}^m\int_{\Ok}\sigma_k V_k\bar{\At^S}\drdz\right|\\
&\leq \scoil\norm{\bV}{\infty}\norm{\At^S}{\Lducoil}\left(\sum_{k=1}^m \IOk \right)^{\frac12},
\end{align*}
and thus
\begin{equation}
\label{cotaAt}
\norm{\At^S(t)}{\Htu}\leq \frac{\scoil}{\tilde{\varrho}}\Big(\sum_{k=1}^m \IOk \Big)^{\frac12}\norm{\bV}{\infty}\qquad\mbox{a.e. }t\in[0,\Tf].
\end{equation}
Now, integrating over $[0,\Tf]$ we obtain the estimate $\norm{\At^S}{\L^2(0,\Tf;\Htu)}$.
\end{proof}

As a consequence of the previous result it follows that the nonlinear operator $\mathcal{H}$
is well-defined and bounded. In particular, for all $S\in\L^2(0,\Tf;\Ldrc)$, $\mathcal{H}(S)$ satisfies
\begin{equation*}%\label{eq:est_H}
\norm{\mathcal{H}(S)}{\L^2(0,\Tf;\Htu)}\leq C_c\norm{\bV}{\infty} .
\end{equation*}

%-------------------------------------------------------------------------------------
Moreover,  operator $\mathcal{H}$ is continuous as stated in the following result.
\begin{lemma}\label{lemma:H_cont}
The mapping $\mathcal{H}:\L^2(0,\Tf;\Ldrc)\to\L^2(0,\Tf;\HtuG)$ is continuous.
\end{lemma}

\begin{proof}
Let $S\in\L^2(0,\Tf;\Ldrc)$ and $\{ S_\ell \}_{\ell\in \N} \subset\L^2(0,\Tf;\Ldrc)$ be a sequence such that $S_\ell\to S$ strongly in $\L^2(0,\Tf;\Ldrc)$. To prove that $\mathcal{H}$ is continuous, we show that the sequence $\{\At^\ell\}_{\ell\in \N}\defeq\{\mathcal{H}(S_\ell)\}_{\ell\in \N}$ converges strongly to $\At^S\defeq\mathcal{H}(S)$ in $\L^2(0,\Tf;\HtuG)$. First we notice that $\At^S$
and $\At^\ell$  satisfy
\begin{align*}
i\omega\int_{\OC}\sigma(\cdot,S) \At^S\bar{Z}r\drdz
+\langle\Acal(\At^S),Z\rangle
=\sum_{k=1}^m\int_{\Ok}\sigma_k V_k\bar{Z}\drdz&&\forall Z\in\HtuG,\\
i\omega\int_{\OC}\sigma(\cdot,S_\ell) \At^\ell\bar{Z}r\drdz
+\langle\Acal(\At^\ell),Z\rangle
=\sum_{k=1}^m\int_{\Ok}\sigma_k V_k\bar{Z}\drdz&&\forall Z\in\HtuG.
\end{align*}
Thus, by subtracting the previous equations and from  the  assumptions on $\sigma$  (cf. \eqref{hyp:sigmak}) we get
\[
i\omega\!\!\int_{\OC} \!\!\! \sigma(\cdot,S_\ell)(\At^\ell-\At^S)\bar{Z}r\drdz
+\langle\Acal(\At^\ell)-\Acal(\At^S),Z\rangle
=-i\omega\!\!\int_{\Ocero}\!\!\!(\sigma(\cdot,S_\ell)-\sigma(\cdot,S))\At^S\bar{Z}r\drdz
\]
for all $Z\in\HtuG$. By taking $Z=\At^\ell-\At^S$ in the previous equation, using \cref{lemma:Acal} (cf. \eqref{AcalSM}) and  by applying H\"older's inequality we
arrive at
\begin{equation}\label{eq:As_Al}
\begin{aligned}
\tilde{\xi}\norm{\At^\ell-\At^S}{\Htu}^2
&\leq
\mathrm{Re}\Big(i\omega\int_{\OC}\sigma(\cdot,S_\ell)|\At^\ell-\At^S|^2r\drdz
+\langle\Acal(\At^\ell)-\Acal(\At^S),\At^\ell-\At^S\rangle\Big)\\
&\leq \Big|i\omega\int_{\Ocero}(\sigma(\cdot,S_\ell)-\sigma(\cdot,S))\At^S\,\overline{\At^\ell-\At^S}r\drdz\Big|\\
&\leq \omega
\Big(\int_{\Ocero}|\sigma(\cdot,S_\ell)-\sigma(\cdot,S)|^2r\drdz\Big)^{\frac12}
\Big(\int_{\Ocero}|\At^S\,(\overline{\At^\ell-\At^S})|^2r\drdz\Big)^{\frac12}\\[2mm]
&\leq\omega\norm{\sigma(\cdot,S_\ell)-\sigma(\cdot,S)}{\Ldrc}\norm{\At^S}{\L^6_1(\Ocero)}\norm{\At^\ell-\At^S}{\L^3_1(\Ocero)}\\
&\leq \omega\,C_e^2\IOcero^{\frac16}C_c\norm{\bV}{\infty}\norm{\sigma(\cdot,S_\ell)-\sigma(\cdot,S)}{\Ldrc}\norm{\At^\ell-\At^S}{\H^1_1(\Ocero)}
\end{aligned}
\end{equation}
where, to obtain the last inequality, we have applied the continuous embedding $\H^1_1(\Ocero) \subset\L^6_1(\Ocero)$ (cf.~\cref{lemma:emb_cont}), the fact that $\norm{W}{\L^3_1(\Ocero)}\leq
\left(\int_{\Ocero}r\drdz\right)^{\frac16}\norm{W}{\L^6_1(\Ocero)}=\IOcero^{\frac16}\norm{W}{\L^6_1(\Ocero)}$ for all $W\in \L^6_1(\Ocero)$ and \eqref{cotaAt}.
Thus, by integrating \eqref{eq:As_Al} over $[0,\Tf]$ it follows that
\[
\int_0^\Tf\norm{\At^\ell(t)-\At^S(t)}{\Htu}^2\dt
\leq \Big(\tilde{\xi}^{-1}\omega\,C_e^2\IOcero^{\frac16}C_c\norm{\bV}{\infty}\Big)^2\int_0^\Tf\norm{\sigma(\cdot,S_\ell(t))-\sigma(\cdot,S(t))}{\Ldrc}^2\dt.
\]
Finally, considering that $\sigma$ is a Carath\'eodory function, by virtue of Lebesgue's dominated convergence theorem we get $\norm{\sigma(\cdot,S_\ell)-\sigma(\cdot,S)}{\L^2(0,\Tf;\Ldrc)}\to 0$ and from the previous estimate we conclude that $\At^\ell\to\At^S$ strongly in $\L^2(0,\Tf;\HtuG)$. Thus we conclude the proof.
\end{proof}

%-------------------------------------------------------------------------------------
% To proceed, we focus on the thermal problem. With this in mind, we consider the following result,
% which will be used in the subsequent discussion.
We next consider the thermal problem. The following result will be used below.

\begin{lemma}\label{lemma:Q_exist}
Let $\bV\in\C^m$, $T_c\in\L^2(0,\Tf;\LdGur)$ and $T_0\in\Ldrc$. For $S\in\L^2(0,\Tf;\Ldrc)$ there exists a unique
$\Sht\in\L^2(0,\Tf;\HuT)\cap\H^1( 0,\Tf;\Ldrc)$ such that
% \begin{align}
% \int_{\Ocero} \partial_t & \Sht W r\drdz
% +\int_{\Ocero}\kappa(\cdot,S)\grad\Sht\cdot\grad Wr\drdz
% +\int_{\widehat{\Gamma}_1}\eta\,\Sht Wr\dl\nonumber\\
% &=\frac12 \int_{\Ocero}\omega^2\sigma(\cdot,S)|\mathcal{H}(S)|^2Wr\drdz
% +\int_{\widehat{\Gamma}_1}\eta\,T_cWr\dl
% \qquad \forall W\in\HuT\quad\text{in }\D'( 0,\Tf),\label{eq:calQa}\\
% \Sht(0)&=T_0\qquad \text{in }\Ocero.\label{eq:calQb}
% \end{align}
%
\begin{align}
\int_{\Ocero} \partial_t \Sht\, W\, r \,\drdz
&+\int_{\Ocero}\kappa(\cdot,S)\grad \Sht\cdot\grad W\, r \,\drdz
+\int_{\widehat{\Gamma}_1}\eta\,\Sht\,W\,r\,\dl \nonumber\\
&= \frac12 \int_{\Ocero}\omega^2\sigma(\cdot,S)|\mathcal{H}(S)|^2\,W\,r\,\drdz
+\int_{\widehat{\Gamma}_1}\eta\,T_c\,W\,r\,\dl,
\label{eq:calQa}\\
&\hspace{8em}\forall W\in\HuT \quad \text{in }\D'(0,\Tf),
\nonumber\\
\Sht(0)&=T_0 \qquad \text{in }\Ocero .
\label{eq:calQb}
\end{align}
Moreover, there exists $\hat{C}>0$ such that
\[
\norm{\Sht}{\L^\infty(0,\Tf;\Ldrc)}+\norm{\Sht}{\L^2(0,\Tf;\HuT)}
\leq \hat{C}\left(\norm{T_0}{\Ldrc}+\norm{\bV}{\infty}^2+\norm{T_c}{\L^2(0,\Tf;\LdGur)}\right)
\]
\end{lemma}
\begin{proof}
The well-posedness of problem \eqref{eq:calQa}--\eqref{eq:calQb} follows from the Poincar\'e-type inequality and trace
inequality in weighted Sobolev given in \cref{lemma:Poincare,lemma:Traza}, respectively, and the classical
theory for parabolic problems (see, for instance, \cite[Chapter~XVIII]{DL}). Now, by taking $W=\Sht(t)$ in
\eqref{eq:calQa}, using  \cref{lemma:Poincare,lemma:Traza}, H\"older's inequality and the fact that $\Huc
\hookrightarrow \L^6_1(\Ocero)$ (cf. \cref{lemma:emb_cont}) and  $\norm{W}{\L^{\frac{12}{5}}_1(\Ocero)}\leq
\left(\int_{\Ocero}r\drdz\right)^{\frac14}\norm{W}{\L^6_1(\Ocero)}$ for all $W\in \L^6_1(\Ocero)$, we obtain
\begin{align*}
\frac{1}{2}\frac{d}{dt}&\norm{\Sht(t)}{\Ldrc}^2
+C_p\min\{\kappa_*,\eta_*\}\norm{\Sht(t)}{\Huc}^2
\\
&\leq \frac{1}{2\Tf}\omega^2\sigma^*C_e^3C_c^2\IOcero^\frac{1}{2}\norm{\bV}{\infty}^2\norm{\Sht(t)}{\Huc}
+\eta^*C_{\mathrm{tr}}\norm{T_c}{\LdGur} \norm{\Sht(t)}{\Huc}\\
& \leq \frac{(C_p\min\{\kappa_*,\eta_*\})^{-1}}{4\Tf^2}\omega^4(\sigma^*)^2C_e^6C_c^4\IOcero\norm{\bV}{\infty}^4
+(C_p\min\{\kappa_*,\eta_*\})^{-1}(\eta^*)^2C_{\mathrm{tr}}^2\norm{T_c(t)}{\LdGur}^2\\
&\quad +\frac{1}{2}C_p\min\{\kappa_*,\eta_*\}\norm{\Sht(t)}{\Huc}^2
\end{align*}
integrating over $[0,t]$, we get
\begin{multline*}
\norm{\Sht(t)}{\Ldrc}^2-\norm{\Sht(0)}{\Ldrc}^2+\int_0^t \norm{\Sht(t)}{\Huc}^2\dt\\
\leq \frac{2(C_p\min\{\kappa_*,\eta_*\})^{-1}}{\min\{1,C_p\min\{\kappa_*,\eta_*\}\}}
\Big(
\frac{1}{4\Tf}\omega^4(\sigma^*)^2C_e^6C_c^4\IOcero\norm{\bV}{\infty}^4
+(\eta^*)^2C_{\mathrm{tr}}^2\int_0^t\norm{T_c(t)}{\LdGur}^2\dt
\Big).
\end{multline*}
Thus, we conclude the proof.
\end{proof}

As a consequence of the previous result, we introduce a nonlinear operator associated with the thermal problem. As in the electromagnetic case, this operator depends on the temperature. More precisely, for $S\in\L^2(0,\Tf;\Ldrc)$, we define a nonlinear
mapping $\mathcal{Q}$ by
$$
\begin{array}{cccll}
&\mathcal{Q}:&\L^2(0,\Tf;\Ldrc)&\longrightarrow&\L^2(0,\Tf;\HuT)\cap\H^1( 0,\Tf;\Ldrc)\\
& & S&\longmapsto&\mathcal{Q}(S)\defeq \Sht
\end{array}
$$
where $\Sht\in\L^2(0,\Tf;\HuT)$ is solution of problem \eqref{eq:calQa}--\eqref{eq:calQb}.
The well-posedness  of $\mathcal{Q}$ follows from \cref{lemma:Q_exist}; its
continuity is proved in the next lemma.
\begin{lemma}\label{lemma:Q_cont}
The mapping $\mathcal{Q}:\L^2(0,\Tf;\Ldrc)\to\L^2(0,\Tf;\HuT)\cap\H^1( 0,\Tf;\Ldrc)$ is continuous.
\end{lemma}
\begin{proof}
First, let us introduce a formulation equivalent of problem \eqref{eq:calQa}--\eqref{eq:calQb}. For
$S\in\L^2(0,\Tf;\Ldrc)$ find $\Sht\in\L^2(0,\Tf;\HuT)\cap\H^1( 0,\Tf;\Ldrc)$ such that for all
$W\in\L^2(0,\Tf;\HuT)\cap\H^1( 0,\Tf;\Ldrc)$ with $W(\Tf)=0$ the following holds

\begin{multline}
\label{eq:calQf}
-\int_0^{\Tf}\int_{\Ocero} \Sht(t) \partial_t W(t) r\drdz\dt
+\int_0^{\Tf}\int_{\Ocero}\kappa(\cdot,S)\grad\Sht(t)\cdot\grad W(t)r\drdz\dt
\\
+\int_0^{\Tf}\int_{\widehat{\Gamma}_1}\eta\,\Sht(t) W(t) r\dl\dt
=\int_0^{\Tf}\int_{\Ocero}\frac12\omega^2\sigma(\cdot,S)|\mathcal{H}(S)(t)|^2W(t)r\drdz\dt
\\
+\int_0^{\Tf}\int_{\widehat{\Gamma}_1}\eta\,T_c(t)W(t)r\dl\dt
+\int_{\Ocero} T_0 W(0) r\drdz.
\end{multline}
Let $S\in\L^2(0,\Tf;\Ldrc)$ be a sequence such that $S_\ell\to S$ strongly in $\L^2(0,\Tf;\Ldrc)$. Let
$\Sht^\ell\defeq\mathcal{Q}(S_\ell)$. We prove that $\mathcal{Q}(S_\ell)\to\mathcal{Q}(S)$ strongly in
$\L^2(0,\Tf;\HuT)\cap\H^1( 0,\Tf;\Ldrc)$. With this aim, we first need to prove that $\Sht^\ell\to\mathcal{Q}(S)$ weakly in
$\L^2(0,\Tf;\HuT)$. Since $\{\Sht^\ell\}_{\ell\in \N}$ is uniformly bounded in
$\L^2(0,\Tf;\HuT)\cap\L^\infty( 0,\Tf;\Ldrc)$ (see \cref{lemma:Q_exist}), it has a weakly convergent (not relabeled)
subsequence in $\L^2(0,\Tf;\HuT)$. Let $\Sht$ be the weak limit of $\{\Sht^\ell\}_{\ell\in \N}$. Next, we recall that
for any $\ell \in \N$, $\Sht^\ell$ satisfy
\begin{multline*}
-\int_0^{\Tf}\int_{\Ocero} \Sht^\ell(t) \partial_t W(t) r\drdz\dt
+\int_0^{\Tf}\int_{\Ocero}\kappa(\cdot,S_\ell)\grad\Sht^\ell(t)\cdot\grad W(t)r\drdz\dt
\\
+\int_0^{\Tf}\int_{\widehat{\Gamma}_1}\eta\,\Sht^\ell(t) W(t) r\dl\dt
=\int_0^{\Tf}\int_{\Ocero}\frac12\omega^2\sigma(\cdot,S_\ell)|\mathcal{H}(S_\ell)(t)|^2W(t)r\drdz\dt
\\
+\int_0^{\Tf}\int_{\widehat{\Gamma}_1}\eta\,T_c(t)W(t) r\dl\dt
+\int_{\Ocero} T_0 W(0) r\drdz
\end{multline*}
for all $W\in\L^2(0,\Tf;\HuT)\cap\H^1( 0,\Tf;\Ldrc)$ with $W(\Tf)=0$. Following the lines of that of Lemma~3.8 from
\cite{GLSV1}, now in time-dependent spaces, we have
\begin{equation}
\label{conv}
\begin{aligned}
-\int_0^{\Tf}\int_{\Ocero} \Sht^\ell(t) \partial_t W(t) r\drdz\dt
&\ \to \
-\int_0^{\Tf}\int_{\Ocero} \Sht(t)\partial_t W(t) r\drdz\dt\\
\int_0^{\Tf}\int_{\Ocero}\kappa(\cdot,S_\ell)\grad\Sht^\ell(t)\cdot\grad W(t)r\drdz\dt
&\ \to \
\int_0^{\Tf}\int_{\Ocero}\kappa(\cdot,S)\grad\Sht(t)\cdot\grad W(t)r\drdz\dt\\
\int_0^{\Tf}\int_{\widehat{\Gamma}_1}\eta\,\Sht^\ell(t) W(t) r\dl\dt
&\ \to \
\int_0^{\Tf}\int_{\widehat{\Gamma}_1}\eta\,\Sht(t) W(t) r\dl\dt \\
\int_0^{\Tf}\int_{\Ocero}\frac12\omega^2\sigma(\cdot,S_\ell)|\mathcal{H}(S_\ell)(t)|^2W(t)r\drdz\dt
&\ \to \
\int_0^{\Tf}\int_{\Ocero}\frac12\omega^2\sigma(\cdot,S)|\mathcal{H}(S)(t)|^2W(t)r\drdz\dt
\end{aligned}
\end{equation}
for all $W\in\L^2(0,\Tf;\HuT)\cap\H^1( 0,\Tf;\Ldrc)$ with $W(\Tf)=0$. Because the equivalence between problem
\eqref{eq:calQa}--\eqref{eq:calQb} and problem \eqref{eq:calQf}, we have proved that $\Sht=\mathcal{Q}(S)$. Moreover,
since the limit is unique, we conclude that $\Sht^\ell\to\Sht$ weakly in $\L^2(0,\Tf;\HuT)$.

Now, to prove that $\{\Sht^\ell\}_{\ell\in \N}$ converges strongly to $\Sht$ in $\L^2(0,\Tf;\HuT)\cap\H^1( 0,\Tf;\Ldrc)$,
we subtract the variational equations verified by $\Sht^\ell$ and $\Sht$, take the test function $W=\Sht^\ell-\Sht$ and
integrating over $[0,t]$, it follows that
\begin{align*}
\frac{1}{2}\norm{(\Sht^\ell-\Sht)(t)}{\Ldrc}^2
&+\int_0^t\int_{\Ocero}\!\!\kappa(\cdot,S_\ell)|\grad(\Sht^\ell-\Sht)(t)|^2r\drdz\dt
+\int_0^t\int_{\widehat{\Gamma}_1}\!\!\eta\,|(\Sht^\ell-\Sht)(t)|^2r\dl\dt
\\
&=\int_0^t\int_{\Ocero}(\kappa(\cdot,S_\ell)-\kappa(\cdot,S))\grad\Sht(t)\cdot\grad(\Sht^\ell-\Sht)(t)r\drdz\dt\\
&\quad+\int_0^t\int_{\Ocero}\!\!\frac12\omega^2
(\sigma(\cdot,S_\ell)|\mathcal{H}(S_\ell)(t)|^2-\sigma(\cdot,S)|\mathcal{H}(S)(t)|^2)(\Sht^\ell-\Sht)(t)r\drdz\dt.
\end{align*}
From \eqref{conv} and the fact that $\{\Sht^\ell\}\to\Sht$ weakly in $\L^2(0,\Tf;\HuT)$, it follows that the right-hand
side of the previous equation converges to zero. Then, the result follows from \eqref{hyp:eta}, \eqref{hyp:sigma} and
\cref{lemma:Poincare}.
\end{proof}
Using the definitions of the operators $\mathcal{H}$ and
$\mathcal{Q}$ together with the previous lemmas, we can prove the existence of the solution of the
nonlinear coupled \cref{pbm:AT_time0} using point-fixed arguments as shown in the following result.

\begin{theorem}
Given $\bV\in\C^m$, $T_c\in\L^2(0,\Tf;\LdGur)$ and $T_0\in\Ldrc$, \cref{pbm:AT_time0} has a solution
$(\At,\Tt)\in\L^2(0,\Tf;\HtuG)\times\L^2(0,\Tf;\HuT)\cap\H^1(0,\Tf;\Ldrc)$.
\end{theorem}
\begin{proof}
Let  $j:\L^2(0,\Tf;\HuT)\cap \H^1(0,\Tf;\Ldrc)\to \L^2(0,\Tf;\Ldrc)$ be the compact injection (see \cite[Lemma~4.2]{MR} and \cite[Chapter~1, Theorem~5.1]{Lions}). Note that
$(\At,\Tt)\in\L^2(0,\Tf;\HtuG)\times\L^2(0,\Tf;\HuT)\cap\H^1(0,\Tf;\Ldrc)$ is solution to \cref{pbm:AT_time0} if and
only if $\Tt$ is a fixed point of the mapping $\mathcal{Q}\circ j$ and $\At=(\mathcal{H}\circ j)(\Tt)$. Since
$\mathcal{Q}$ is continuous mapping and $j$ is compact, $\mathcal{Q}\circ j$ is compact. Moreover, since
$\norm{\mathcal{Q}(\Tt)}{\L^2(0,\Tf;\HuT)} \leq C(T_0,T_c,\bV)$, then Schauder's theorem yields the existence of a fixed
point $\Tt$ of $\mathcal{Q}$ and the proof is complete.
\end{proof}

%-------------------------------------------------------------------------------------
\subsection{Uniqueness}
Since \cref{pbm:AT_time0} is nonlinear, we do not known whether the solution is unique. In fact, the study of
uniqueness of solution, without making additional assumptions on the data, is a difficult task. Therefore, for the
following analysis we will further assume that the coefficients $\kappa$ and $\sigma$ satisfy:
\begin{itemize}
\item[$\mathbf{H}_{\kappa}.$]
There exists $\kcte>0 $ such that $\kappa(r,z,u)=\kcte(r,z)$ in $\Ocero\times \R$ and
$\kappa_*\leq\kcte(r,z)\leq\kappa^*$ in $\Ocero$.

\item[$\mathbf{H}_{\sigma}.$] $\sigma$ is uniformly Lipschitz continuous in the following sense: there exists
$\sigma_p>0$ such that
\begin{equation}\label{hyp:Lsigma}
|\sigma(r,z,u)-\sigma(r,z,v)|\leq \sigma_p|u-v|\quad \forall u,v\in \R, \mbox{ a.e. } (r,z) \in \Ocero.
\end{equation}
\end{itemize}

From hypothesis $\mathbf{H}_{\kappa}$ and \cref{pbm:AT_time0} we get the following variational formulation.
Notice that we have removed the subscript $\theta$ to simplify the notation.

\begin{problem}\label{pbm:AT_time}
Given $\bV\in\C^m$, $T_c\in\L^2(0,\Tf;\LdGur)$ and $T_0\in\Ldrc$, find $(A,T)\in\L^2(0,\Tf;\HtuG)\times
\L^2(0,\Tf;\HuT)\cap\H^1(0,\Tf;\Ldrc)$ such that for a.e. $t\in(0,\Tf)$
\begin{align*}
&i\omega\int_{\OC}\sigma(\cdot,T) A\,\bar{Z}\,r\drdz
+\langle\Acal(A),Z\rangle
=\sum_{k=1}^m\int_{\Ok}\sigma_k V_k\bar{Z}\drdz,
&& \forall Z\in\HtuG,
\\
&\int_{\Ocero}\dpt T\, W\, r\drdz
+b(T,W)
=\frac12\int_{\Ocero}\omega^2\sigma(\cdot,T)|A|^2Wr\drdz
+\int_{\Gtu}\eta\,T_c\,Wr\dl,
&& \forall W\in\HuT,
\\
&T(0)=T_0\qquad \mbox{in }\Ocero,
\end{align*}
where $\Acal$ is defined in \eqref{eq:Acal} and $b:\HuT\times\HuT \to \R$ is such that
\[
b(T,W)\defeq\int_{\Ocero}\kcte\grad T\cdot\grad W r\drdz+\int_{\widehat{\Gamma}_1}\eta\,T\, W\, r\dl.
\]
\end{problem}

%\begin{problem}\label{AT_time}
%Given $\bV\in\C^m$, $T_c\in\L^2(0,\Tf;\LdGur)$ and $T_0\in\Ldrc$, find $(A,T)\in\L^2(0,\Tf;\HtuG)\times
%\L^2(0,\Tf;\HuT)\cap\H^1(0,\Tf;\Ldrc)$ such that for all $(Z,W)\in\HtuG\times\HuT$ and a.e. $t\in(0,\Tf)$
%\begin{align*}
%i\omega\int_{\OC}\sigma(\cdot,T) A\,\bar{Z}\,r\drdz
%+\langle\Acal(A),Z\rangle
%&=\sum_{k=1}^m\int_{\Ok}\sigma_k V_k\bar{Z}\drdz,
%\\
%\int_{\Ocero}\dpt T\, W\, r\drdz
%+b(T,W)
%&=\frac12 \int_{\Ocero}\omega^2\sigma(\cdot,T)|A|^2Wr\drdz
%+\int_{\Gtu}\eta\,T_c\,Wr\dl, \\
%T(0)&=T_0\qquad \mbox{in }\Ocero,
%\end{align*}
%where $\Acal$ is defined in \eqref{eq:Acal} and $b:\HuT\times\HuT \to \R$ is such that
%\[
%b(T,W)\defeq\int_{\Ocero}\kcte\grad T\cdot\grad W r\drdz+\int_{\widehat{\Gamma}_1}\eta\,T\, W\, r\dl.
%\]
%\end{problem}

In the following theorem, given the previous assumptions, we establish  uniqueness of the solution to \cref{pbm:AT_time}.
\begin{theorem}\label{U_AT}
Given $\bV\in\C^m$, $T_c\in\L^2(0,\Tf;\LdGur)$ and $T_0\in\Ldrc$, if $\mathbf{H}_{\kappa}$ and $\mathbf{H}_{\sigma}$ are
satisfied, then \cref{pbm:AT_time} has a unique solution.
\end{theorem}
\begin{proof}
Let $(A_1,T_1)$ and $(A_2,T_2)$ be two solutions of \cref{pbm:AT_time}. By proceeding as in the proof of
\cref{lemma:H_cont}  (cf.~\eqref{eq:As_Al}) and from $\mathbf{H}_{\sigma}$ it follows that there exists $C>0$ such that
\begin{equation}\label{eq:A_Al}
\norm{A_1-A_2}{\Htu}
%\leq C\norm{\sigma(T_1)-\sigma(T_2)}{\Ldrc}
\leq C \norm{T_1-T_2}{\Ldrc}.
\end{equation}
Next we estimate $\norm{T_1-T_2}{\Ldrc}$. From the thermal equation in \cref{pbm:AT_time} we get
\begin{equation*}
\int_{\Ocero}\partial_t(T_1-T_2) Wr\drdz
+b(T_1- T_2,W)
%+\int_{\Ocero}\kcte\grad (T_1- T_2)\cdot \grad W r\drdz
%+\int_{\widehat{\Gamma}_1}\eta(T_1-T_2)Wr\dl\\
=\frac12 \int_{\Ocero}\omega^2(\sigma(T_1)|A_1|^2-\sigma(T_2)|A_2|^2)\,W\,r\drdz.
%\label{eq:T1_T2}
\end{equation*}
By taking $W=T_1-T_2$ in the previous equation, from $\mathbf{H}_{\kappa}$ and \cref{lemma:Poincare}  we obtain
\begin{align}
\frac{1}{2}
\frac{d}{dt}&\norm{T_1-T_2}{\Ldrc}^2
+C_p\min\{\kappa_*,\eta_*\}\norm{T_1-T_2}{\Huc}^2\nonumber\\
&\leq \frac{\omega^2}{2}\int_{\Ocero}(\sigma(T_1)|A_1|^2
-\sigma(T_2)|A_2|^2)(T_1-T_2)\,r\drdz\nonumber\\
&\leq \frac{\omega^2}{2}\left[
\norm{(\sigma(T_1)-\sigma(T_2))|A_2|^2}{\Ldrc}
+\norm{\sigma(T_1)(|A_1|^2-|A_2|^2)}{\Ldrc}
\right]
\norm{T_1-T_2}{\Ldrc}.\label{eq:T1_T2_b}
\end{align}
It remains to estimate the right-hand side of the previous equation. For these terms we apply \eqref{hyp:sigma}, \eqref{hyp:Lsigma} and H\"older's inequality, then there exists $C>0$ such that
\begin{align}
\label{eq:T1_T2_A2}
\norm{(\sigma(T_1)-\sigma(T_2))|A_2|^2}{\Ldrc}
&\leq C\norm{T_1-T_2}{\L^6_1(\Ocero)}\norm{A_2}{\L_1^6(\Ocero)}^2\leq C\norm{T_1-T_2}{\HuT},
\\
\norm{\sigma(T_1)(|A_1|^2-|A_2|^2)}{\Ldrc}
&\leq C\norm{A_1+A_2}{\L^3_1(\Ocero)}\norm{A_1-A_2}{\L^6_1(\Ocero)}\leq C\norm{A_1-A_2}{\Htu}\label{eq:T1_T2_e}
\end{align}
where we have used \cref{lemma:emb_cont} and the fact that $\norm{A_i}{\Htu}\leq \frac{C_c}{\sqrt{\Tf}} \norm{\bV}{\infty}, i \in \{1,2\}$ (cf. \eqref{cotaAt}).

By substituting \eqref{eq:T1_T2_A2}--\eqref{eq:T1_T2_e} in \eqref{eq:T1_T2_b} and using the bound for
$\norm{A_1-A_2}{\Htu}$ (cf.~\eqref{eq:A_Al}), it follows that there exists   $C>0$ such that
\begin{equation*}
\frac{d}{dt}\norm{T_1-T_2}{\Ldrc}^2
+ \norm{T_1-T_2}{\Huc}^2\\
\leq
C\norm{T_1-T_2}{\Ldrc}^2\ .
\end{equation*}
Therefore, from Gr\"{o}nwall's inequality  it follows that $T_1=T_2$ which  implies that
$A_1=A_2$ (cf.~\eqref{eq:A_Al}) and the proof is complete.
\end{proof}

%-------------------------------------------------------------------------------------
\begin{remark}  We assume $\mathbf{H}_{\kappa}$ for the sake of clarity in presenting the uniqueness result,
although it is important to note that the analysis extends to a more general case, e.g. (see \cite{GLSV1})
\begin{itemize}
\item[$\mathbf{H}^*_{\kappa}.$] $\kappa(r,z,u)=\alpha(r,z)\beta(u)$ where $\alpha$ and
$\beta$ are measurable and continuous functions, respectively, that  satisfy
$0<\alpha_*\leq\alpha\leq\alpha^*$ and  $0<\beta_*\leq\beta\leq\beta^*$,
$\alpha_*,\alpha^*,\beta_*,\beta^* \in \R$.
\end{itemize}
%The previous uniqueness result can be proved by considering
%less restictive assumptions for $\kappa$, e.g.
%In order to present the result briefly, we have decided to assume
% which is used to study the discrete problem as in \cite{GLSV1}.
\end{remark}

%-------------------------------------------------------------------------------------
\section{Fully discrete approximation}\label{sec:fully}
%-------------------------------------------------------------------------------------

In this section, we introduce a space discretization based on finite elements and a backward Euler scheme for time
discretization of \cref{pbm:AT_time} and obtain error estimates for the proposed approximation.

Let $\left\{\Th\right\}_{h>0}$ be a regular family of partitions of $\bar{\O}=[0,R]\times[0,L]$ in triangles, where $h$
denotes the mesh-size (i.e., the maximal length of the sides of the triangulation). We also assume that each element
$K\in \Th$ is contained either in $\OA$ or in $\Ok$, $k=0,\dots,m$. We denote
\[
\Th^{0}\defeq\left\{K\in\Th:K\subset\Occ\right\}.
\]

Let $\Y_h$ be the space of piecewise linear continuous finite elements with vanishing values in $\Gamma$,
\[\Y_h\defeq\{Z\in \mathcal{C}(\bar{\O})\ :\ Z|_K\in \mathbb{P}_1(K)\ \ \forall K\in\Th,\ Z|_{\Gamma}=0\}\subset \HtuG\]
and
\[
\X_h\defeq \{W\in \mathcal{C}(\Occ)\ :\ W|_K\in \mathbb{P}_1(K)\ \ \forall K\in\Th^{0}\}\subset \HuT.
\]
For the time discretization, we consider a uniform partition of $[0,\Tf]$, $t_n:=n\Dt$, $n=0,\dots,N$, with time step $\Dt :=\Tf/N$. Let us now introduce the following notation: if $v$ is
regular enough with respect to $t$, we denote $v^n := v (t_n), n = 0, \ldots  N$. On the other hand, for any
sequence $\{u^{n}\}_{n=0}^N$, we define
\[
\diff u^{n+1}\defeq \frac{u^{n+1}-u^{n}}{\Dt},\qquad n=0,\dots,N-1.
\]
We write  $\texttt{a} \lesssim \texttt{b}$ to mean that $\texttt{a} \leq C\, \texttt{b}$, where $C>0$ is a generic constant independent of $\texttt{a}$, $\texttt{b}$, the mesh size $h$ and the time step $\Dt$. The value of $C$ may change at each occurrence and will be specified only when is needed.

For the numerical scheme, we consider that $\bV$ and $T_c$ are continuous in time and that we dispose of an
approximation $T_h^0\in \X_h$ of the initial data $T_0$. We propose the following fully discrete approximation of \cref{pbm:AT_time}:
%
%\begin{problem}\label{AT_fully}
%Given $\bV\in\C^m$, $T_c\in\Ccal([0,\Tf];\LdGur)$ and $T_h^0\in \X_h$, find
%$(A_h^n,T_h^{n+1})\in\Y_h\times\X_h,\ n = 0,\ldots,N$, such that for all $(Z_h,W_h)\in \Y_h\times\X_h$
%\begin{align*}
%i\omega\int_{\OC}\sigma(T_h^n) A_h^n\bar{Z}_h r\drdz
%+\int_\O \nu(|\curl(A_h^n\be_{\theta})|)
%\left(
%\frac{1}{r}\frac{\partial(rA_h^n)}{\partial r}\frac{1}{r}\frac{\partial(r\bar{Z}_h)}{\partial r}
%+\frac{\partial A_h^n}{\partial z}\frac{\partial \bar{Z}_h}{\partial z}\right)r\drdz
%\\
%=\sum_{k=1}^m\int_{\Ok}\sigma_k\,V_k\bar{Z}_h\drdz\\
%\int_{\Ocero}\diff T_h^{n+1}\, W_h r\drdz
%+\int_{\Ocero}\kcte\grad T_h^{n+1}\cdot\grad W_h r\drdz
%+\int_{\Gtu}\eta\,T^{n+1}_h W_h r\dl\\
%+b(T_h^{n+1},W_h)
%=\int_{\Ocero}\omega^2\sigma(T_h^n)|A_h^n|^2 W_h r\drdz
%+\int_{\Gtu}\eta\,T_c^n W_h r\dl .
%\end{align*}
%\end{problem}
\begin{problem}\label{pbm:AT_fully}
Let $\bV\in\C^m$, $T_c\in\Ccal([0,\Tf];\LdGur)$ and $T_h^0\in \X_h$ be given data.
\begin{itemize}[leftmargin=*]
\item[$\scriptscriptstyle \bullet$]  Find $A_h^0\in\Y_h$ such that
\begin{equation}\label{eq:fully_A0}
i\omega\int_{\OC}\sigma(T_h^{0}) A_h^0\,\bar{Z}_h \, r\drdz
+ \langle\Acal(A^0_h),Z_h \rangle
%+\int_\O \nu(|\curl(A_h^0\be_{\theta})|)
%\left(
%\frac{1}{r}\frac{\partial(rA_h^0)}{\partial r}\frac{1}{r}\frac{\partial(r\bar{Z}_h)}{\partial r}
%+\frac{\partial A_h^0}{\partial z}\frac{\partial \bar{Z}_h}{\partial z}\right)r\drdz\\
=\sum_{k=1}^m\int_{\Ok}\sigma_k\,V_k\bar{Z}_h\drdz \quad  \forall Z_h\in \Y_h
\end{equation}
\item[$\scriptscriptstyle \bullet$] For   $\,n =0,\ldots,N-1$, let $(A_h^{n+1},T_h^{n+1})\!\in\!\Y_h\!\times\X_h$ such that, for all $(W_h,Z_h)\!\in\!\Y_h\times\X_h$
\begin{align}
\int_{\Ocero}\diff T_h^{n+1}\, W_h\, r\drdz+b(T_h^{n+1},W_h)
=\frac12 \int_{\Ocero}\omega^2\sigma(T_h^{n})|A_h^{n}|^2 W_h\, r\drdz
+\int_{\Gtu}\eta\,T_c^{n} W_h\, r\dl, \label{eq:fully_Tn}\\
i\omega\int_{\OC}\sigma(T_h^{n+1}) A_h^{n+1}\bar{Z}_h\, r\drdz
+ \langle\Acal(A^{n+1}_h),Z_h \rangle
%+\int_\O \nu(|\curl(A_h^{n+1}\be_{\theta})|)
%\left(
%\frac{1}{r}\frac{\partial(rA_h^{n+1})}{\partial r}\frac{1}{r}\frac{\partial(r\bar{Z}_h)}{\partial r}
%+\frac{\partial A_h^{n+1}}{\partial z}\frac{\partial \bar{Z}_h}{\partial z}\right)r\drdz\nonumber\\
=\sum_{k=1}^m\int_{\Ok}\sigma_k\,V_k\bar{Z}_h\drdz.\label{eq:fully_An}
\end{align}
\end{itemize}
\end{problem}
We can prove the existence and uniqueness of $A_h^{n}$, for each $n=0,\ldots,N$, using \cref{lemma:Htu,lemma:Acal}, along with a known existence result for nonlinear problems (see \cite[Theorem 25.B]{zeidler2b}). On the other hand, the existence and uniqueness of  $T_h^{n}$, $n=1,\ldots,N$, follow directly from \cref{lemma:Poincare} due to the Lax-Milgram theorem. Thus, \cref{pbm:AT_fully} is well posed. Moreover, the solution is uniformly bounded in $n$, in particular
\begin{equation}\label{eq:est_ATh}
\begin{aligned}
\norm{A_h^n}{\Htu}&\lesssim\norm{\bV}{\infty}\\
\norm{T_h^{n+1}}{\Ldrc}+\Big(\Dt\sum_{k=1}^n\norm{T_h^k}{\Huc}^2\Big)^{1/2}&\lesssim \norm{T_h^{0}}{\Ldrc}+\norm{\bV}{\infty}^2+\norm{T_c}{\Ccal([0,\Tf];\LdGur)}.
\end{aligned}
\end{equation}

%For the analysis of convergence we define the sesquilinear form $c:\HtuG\times\HtuG\to \C$ such that
%\[
%c(A ,Z)\defeq \int_\O \frac{1}{\mu}\left(
%\frac{1}{r}\frac{\partial(rA)}{\partial r}\frac{1}{r}\frac{\partial(r\bar{Z})}{\partial r}
%+\frac{\partial A}{\partial z}\frac{\partial \bar{Z}}{\partial z}
%\right)r\drdz\quad\forall A,Z \in \HtuG.
%\]
%$R_h Z \in \Y_h$ satisfies
%\[
%c(Z-R_hZ,Z_h)=0 \quad \forall Z_h\in\Y_h.
%\]
%Since the sesquilinear form $c$ is $\HtuG$-elliptic and continuous (cf.~\cref{lemma:Htu}) it follows that $R_h$ is
%well-defined and there exists $\CR\geq 0$ such that
%
 Let  $I_h:\HtuG\to\Y_h$ and
$Q_h:\HuT\to\X_h$ be  Cl\'ement-type operators such that, for all $s\in [1,2]$ (see \cite[Theorem~1 and Theorem~2]{BBD2006}),
\begin{align}\label{est_Ih}
\norm{Z-I_hZ}{\Htu}&\lesssim h^{s-1}\norm{Z}{\H^s_1(\O)\cap\Htu}, && Z\in \H^s_1(\O)\cap\HtuG,\\
h^{-1}\norm{W-Q_hW}{\Ldrc}+\norm{W-Q_hW}{\Huc}&\lesssim h^{s-1}\norm{W}{\H^s_1(\Ocero)}, && W\in \H^s_1(\Ocero).\label{est_Qh}
\end{align}
Our next goal is to derive error estimates for \cref{pbm:AT_fully},
the numerical approximation of \cref{pbm:AT_time}. To this end, we henceforth assume that the solution to
\cref{pbm:AT_time} satisfies $(A,T)\in\Ccal([ 0,\Tf];\HtuG)\times\Ccal([ 0,\Tf];\Huc)$. For $n = 0,\ldots,N$, we define the errors and
split them as follows
\[
A_h^n-A^n\defeq \eA^n+I_h A^n-A^n,\qquad T_h^n-T^n\defeq \eT^n+Q_h T^n-T^n,
%\qquad n = 0,\ldots,N,
\]
where
\[
\eA^n\defeq A_h^n-I_h A^n,\qquad \eT^n\defeq T_h^n-Q_h T^n.% ,\qquad n = 0,\ldots,N.
\]

Now, we are in a position to write the main result of this paper related to the convergence of the proposed scheme.
\begin{theorem}
Given $\bV\in\C^m$, $T_c\in\H^1(0,\Tf;\LdGur)$ and  $T_0\in\HuT$. Let
$(A,T)\in\L^2(0,\Tf;\HtuG)\times \L^2(0,\Tf;\HuT)\cap\H^1(0,\Tf;\Ldrc)$, $A_h^0\in\Y_h$ and $(A_h^n,T_h^n)\in\Y_h\times\X_h,\ n =
1,\ldots,N$ solutions to \cref{pbm:AT_time} and \cref{pbm:AT_fully}, respectively. If
$(A,T)\in\H^1(0,\Tf;\Huc)\times\H^2(0,\Tf;\Ldrc)$ and
$T_h^0:=Q_hT^0$, then
\begin{align*}
\underset{1\leq n\leq N}{\max}&\norm{T_h^n-T^n}{\Ldrc}^2
+\Dt\sum_{n=1}^N\norm{T_h^n-T^n}{\Huc}^2
%+\Dt\sum_{n=1}^N\norm{A_h^n-A^n}{\Htu}^2
+\underset{0\leq n\leq N}{\max}\norm{A_h^n-A^n}{\Htu}^2
\\
&\lesssim
\underset{0\leq n\leq N}{\max}\norm{A^n-I_hA^n}{\Htu}^2
+\underset{0\leq n\leq N}{\max}\norm{T^n-Q_h T^n}{\Huc}^2
+\norm{(I-Q_h)\partial_t T}{\L^2(0,\Tf;\Ldrc)}^2
\\
&\quad+\Dt^2\left(
\norm{\partial_t T_c}{\L^2(0,\Tf;\LdGur)}^2
+\norm{ T}{\H^2(0,\Tf;\Ldrc)}^2
+\norm{\partial_t A}{\L^2(0,\Tf;\Huc)}^2
\right).
\end{align*}
\end{theorem}

\begin{proof}
Let us begin by estimating the error in the electromagnetic problem. To this end, we note that, for each time $t_n$, $n=0,\dots,N$, the pair $(A^n,T^n)$ satisfies
\begin{equation}
\label{A_tn}
i\omega\int_{\OC}\sigma(T^n) A^n\bar{Z}\,r\drdz
+\langle\Acal(A^n),Z\rangle
=\sum_{k=1}^m\int_{\Ok}\sigma_k V_k\bar{Z}\drdz
\end{equation}
for all $Z\in\HtuG$. By subtracting $\eqref{A_tn}$ from \eqref{eq:fully_A0} and \eqref{eq:fully_An},
 using the definitions of $I_h$ and $\eA^n$ it follows that, for all $Z_h\in\Y_h$
\begin{multline*}
i\omega\int_{\OC}\sigma(T_h^n) \eA^n\bar{Z}_hr\drdz
+\langle\Acal(A_h^n) -\Acal(I_h A^n),Z_h\rangle
=i\omega\int_{\Ocero}\left(\sigma(T^n)-\sigma(T_h^n)\right) A^n\bar{Z}_hr\drdz\\
+i\omega\int_{\OC}\sigma(T_h^n)\left(A^n-I_h A^n\right)\bar{Z}_hr\drdz
+\langle\Acal(A^n) -\Acal(I_h A^n),Z_h\rangle.
\end{multline*}
By taking $Z_h=\eA^n$ in the previous equation, using \eqref{hyp:sigma}, \eqref{hyp:Lsigma}, \cref{lemma:Acal} and  applying Cauchy-Schwarz inequality we arrive at
\begin{align*}
\tilde{\varrho}\norm{\eA^n}{\Htu}^2&
\leq \mathrm{Re}\Big(i\omega\int_{\OC}\sigma(T_h^n) \left|\eA^n\right|^2r\drdz
+\langle\Acal(A_h^n) -\Acal(I_h A^n),\eA^n\rangle\Big)
\\
&\leq \omega\sigma_p \norm{T^n-T_h^n}{\Ldrc}\norm{A^n\eA^n}{\Ldrc}\\
% + \omega\max\{\sigma^*,\scoil\}\norm{A^n-I_h A^n}{\LdrOC}\norm{\eA^n}{\LdrOC}\\
&\quad+ \omega\max\{\sigma^*,\scoil\}\norm{A^n-I_h A^n}{\LdrOC}\norm{\eA^n}{\LdrOC}\\
&\quad+\tilde{\xi}\norm{A^n-I_hA^n}{\Htu}\norm{\eA^n}{\Htu}\\
&\lesssim \left(\norm{T^n-T_h^n}{\Ldrc}
+\norm{A^n-I_hA^n}{\Htu}
\right)\norm{\eA^n}{\Htu}
\end{align*}
where, in the last inequality, we have estimated the first term on the right-hand side by proceeding as in \eqref{eq:As_Al}.
From the definition of $\eT^n$ and the previous equation we get
\begin{equation}
\label{etanh}
\norm{\eA^n}{\Htu}\lesssim \norm{\eT^n}{\Ldrc}+\norm{Q_hT^n-T^n}{\Ldrc}+\norm{A^n-I_h A^n}{\Htu}.
\end{equation}

Let us now estimate the error for the temperature. Notice that $(A^{n+1},T^{n+1})$ satisfies the following equations
\begin{equation*}
\int_{\Ocero}\partial_t T^{n+1} Wr\drdz+
b(T^{n+1},W)
=\frac12 \int_{\Ocero}\omega^2\sigma(T^{n+1})|A^{n+1}|^2Wr\drdz
+\int_{\Gtu}\eta\,T_c^{n+1}\,Wr\dl
\end{equation*}
for all $W\in\HuT$. Moreover, the previous equality can be rewritten as follows:
\begin{equation}
\label{T_tnp1N}
\begin{aligned}
\int_{\Ocero}\diff T^{n+1}\,W\,r\drdz
+b(T^{n+1},W)
&=\frac12\int_{\Ocero}\omega^2\sigma(T^{n})|A^{n}|^2\,W\,r\drdz \\
&\quad +\int_{\Gtu}\eta\,T_c^{n+1}\,W\,r\dl
+\mathcal{R}^n(W),
\end{aligned}
\end{equation}
where
% \begin{multline*}
% \mathcal{R}^n(W)\defeq \int_{\Ocero}\left(\diff T^{n+1}-\partial_t T^{n+1}\right)Wr\drdz
% +\frac12 \int_{\Ocero}\omega^2\sigma(T^n)\left(|A^{n+1}|^2-|A^n|^2\right)Wr\drdz\\
% +\frac12\int_{\Ocero}\omega^2\left(\sigma(T^{n+1})-\sigma(T^n)\right)|A^{n+1}|^2Wr\drdz.
% \end{multline*}
\begin{equation*}
\begin{aligned}
\mathcal{R}^n(W)\defeq\;&
\int_{\Ocero}\left(\diff T^{n+1}-\partial_t T^{n+1}\right)W\,r\drdz \\
&+\frac12\omega^2\int_{\Ocero}\sigma(T^n)\left(|A^{n+1}|^2-|A^n|^2\right)W\,r\drdz \\
&+\frac12\omega^2\int_{\Ocero}\left(\sigma(T^{n+1})-\sigma(T^n)\right)|A^{n+1}|^2W\,r\drdz .
\end{aligned}
\end{equation*}
We estimate $\mathcal{R}^n(W)$ by using \eqref{hyp:sigma}, Cauchy-Schwarz  inequality and by proceeding as in \eqref{eq:T1_T2_A2} and \cref{lemma:emb_cont} to estimate the last term
% \begin{multline*}
% |\mathcal{R}^n(W)|\lesssim
% \left(
% \norm{\diff T^{n+1}-\partial_t T^{n+1}}{\Ldrc}
% +\norm{|A^{n+1}|^2-|A^n|^2}{\Ldrc}\right)\norm{W}{\Ldrc}\\
% +\norm{\sigma(T^{n+1})-\sigma(T^n)}{\Ldrc}\norm{A^{n+1}}{\Huc}^2\norm{W}{\Huc}.
% \end{multline*}
\begin{equation*}
\begin{aligned}
|\mathcal{R}^n(W)|
\lesssim\;&
\Big(
\norm{\diff T^{n+1}-\partial_t T^{n+1}}{\Ldrc}
+\norm{|A^{n+1}|^2-|A^n|^2}{\Ldrc}
\Big)\norm{W}{\Ldrc}\\
&+\norm{\sigma(T^{n+1})-\sigma(T^n)}{\Ldrc}\,
\norm{A^{n+1}}{\Huc}^2\,\norm{W}{\Huc}.
\end{aligned}
\end{equation*}
Next, we bound each term on the right-hand side of the previous equation by using \eqref{hyp:Lsigma}, \eqref{cotaAt}, \cref{lemma:emb_cont} and applying H\"older's inequality and Taylor's formula
\begin{align*}
\norm{\diff T^{n+1}-\partial_t T^{n+1}}{\Ldrc}^2
%&=\Bnorm{-\frac{1}{\Dt}\int_{t_n}^{t_{n+1}}(s-t_n)\partial_{tt} T(s)\ds}{\Ldrc}^2\\
&\leq \Dt\int_{t_n}^{t_{n+1}}\norm{\partial_{tt} T(s)}{\Ldrc}^2\ds
\\
\norm{|A^{n+1}|^2-|A^n|^2}{\Ldrc}^2
&\lesssim \norm{A^{n+1}+A^n}{\L_1^3(\Ocero)}^2\norm{A^{n+1}-A^n}{\L_1^6(\Ocero)}^2
\lesssim \Dt\int_{t_n}^{t_{n+1}}\!\!\!\norm{\partial_t A(s)}{\Huc}^2\ds
\\
%\norm{\left(\sigma(T^{n+1})-\sigma(T^n)\right)|A^{n+1}|^2}{\Ldrc}^2
%&\lesssim \norm{A^{n+1}}{\L_1^3(\Ocero)}^4\norm{T^{n+1}-T^n}{\L_1^6(\Ocero)}^2\\
%&\lesssim \Dt\int_{t_n}^{t_{n+1}}\norm{\partial_t T(s)}{\Huc}^2\ds.
\norm{\sigma(T^{n+1})-\sigma(T^n)}{\Ldrc}^2
&\lesssim
\Dt\int_{t_n}^{t_{n+1}}\norm{\partial_t T(s)}{\Ldrc}^2\ds.
\end{align*}
From the previous equations and Young's inequality,  there exists $\epsilon_1 > 0$ such that
\begin{equation}\label{RW}
|\mathcal{R}^n(W)|\lesssim
\dfrac{\Dt}{\epsilon_1}\int_{t_n}^{t_{n+1}}\left\{
\norm{\partial_{tt} T(s)}{\Ldrc}^2
+\norm{\partial_t A(s)}{\Huc}^2
+\norm{\partial_t T(s)}{\Ldrc}^2
\right\}\ds
+\epsilon_1\norm{W}{\Huc}^2.
\end{equation}

On the other hand, subtracting $\eqref{T_tnp1N}$ from \eqref{eq:fully_Tn}, using the
definitions of $I_h$, $Q_h$, $\eA^n$ and $\eT^{n+1}$ it follows that
% \begin{align*}
% \int_{\Ocero}\diff \eT^{n+1} W_hr\drdz+
% b(\eT^{n+1},W_h)
% =\int_{\Ocero}\diff(T^{n+1}-Q_hT^{n+1}) W_hr\drdz
% +b(T^{n+1}-Q_hT^{n+1},W_h)\\
% +\frac12 \int_{\Ocero}\omega^2\left(\sigma(T_h^n)-\sigma(T^n)\right)|A_h^n|^2W_hr\drdz
% +\frac12 \int_{\Ocero}\omega^2\sigma(T^n)\left(|A_h^n|^2-|A^n|^2\right)W_hr\drdz\\
% +\int_{\Gtu}\eta\,\left(T_c^n-T_c^{n+1}\right)\,W_hr\dl-\mathcal{R}^n(W_h).
% \end{align*}
\begin{align*}
\int_{\Ocero}\!\!\diff \eT^{n+1} W_h\,r\drdz
+b(\eT^{n+1},W_h)
&=\int_{\Ocero}\!\!\diff(T^{n+1}\!-Q_hT^{n+1})\,W_h\,r\drdz
+b(T^{n+1}\!-\!Q_hT^{n+1},W_h)\\
&\quad +\frac12\int_{\Ocero}\omega^2\left(\sigma(T_h^n)-\sigma(T^n)\right)|A_h^n|^2W_h\,r\drdz\\
&\quad +\frac12\int_{\Ocero}\omega^2\sigma(T^n)\left(|A_h^n|^2-|A^n|^2\right)W_h\,r\drdz\\
&\quad +\int_{\Gtu}\eta\,\left(T_c^n-T_c^{n+1}\right)\,W_h\,r\dl
-\mathcal{R}^n(W_h).
\end{align*}
By taking $W_h=\eT^{n+1}$ in the previous equation, using \eqref{hyp:sigma}, \eqref{hyp:Lsigma}, \eqref{eq:est_ATh} and by applying
Cauchy-Schwarz and Young's inequalities we obtain
\begin{align*}
\frac{1}{2\Dt}&\left(\norm{\eT^{n+1}}{\Ldrc}^2-\norm{\eT^n}{\Ldrc}^2\right)
+C_p\min\{\mathring{\kappa},\eta_*\}\norm{\eT^{n+1}}{\Huc}^2
\\
\lesssim& |\mathcal{R}^n(\eT^{n+1})|
+\dfrac{1}{\epsilon_2}\left\{
\norm{\diff(T^{n+1}-Q_hT^{n+1})}{\Ldrc}
+\norm{T_c^n-T_c^{n+1}}{\LdGur}
\right.
\\
&%\hspace*{3cm}
\left.
+\norm{T^{n+1}-Q_hT^{n+1}}{\Huc}
+\norm{T_h^n-T^n}{\Ldrc}
+\norm{A_h^n-A^n}{\Huc}
\right\}^2+\epsilon_2\norm{\eT^{n+1}}{\Huc}^2
\end{align*}
for some $\epsilon_2 >0$. By estimating the second and third term on the right-hand side of the previous equation
in a similar way as was done to obtain \eqref{RW}, from  \eqref{RW}, \eqref{etanh}  and taking $\epsilon_1=\epsilon_2=C_p\min\{\mathring{\kappa},\eta_*\}/4$
it follows that
\begin{align*}
\frac{1}{2\Dt}&\left(\norm{\eT^{n+1}}{\Ldrc}^2-\norm{\eT^n}{\Ldrc}^2\right)
+\dfrac{C_p\min\{\mathring{\kappa},\eta_*\}}{2} \norm{\eT^{n+1}}{\Huc}^2\\
&\lesssim
\norm{\eT^n}{\Ldrc}^2
+\norm{A^n-I_h A^n}{\Htu}^2\\
&\quad +\norm{Q_hT^n-T^n}{\Ldrc}^2
+\frac{1}{\Dt}\int_{t_n}^{t_{n+1}}\norm{(I-Q_h)\partial_t T(s)}{\Ldrc}^2\ds
+\norm{T^{n+1}-Q_hT^{n+1}}{\Huc}^2\\
&\quad +\Dt\int_{t_n}^{t_{n+1}}
\left\{
\norm{\partial_t T_c(s)}{\LdGur}^2
+\norm{\partial_{tt} T(s)}{\Ldrc}^2
+\norm{\partial_t A(s)}{\Huc}^2
+\norm{\partial_t T(s)}{\Ldrc}^2
\right\}\ds.
\end{align*}
In particular, the term $\norm{\eT^n}{\Ldrc}$ is derived from
$\norm{T_h^n-T^n}{\Ldrc}$, $\norm{A_h^n-A^n}{\Huc}$ adding and subtracting $Q_hT^n$, $I_hA^n$  and
using \eqref{etanh}.
Then, multiplying by $2\Delta t$, summing over $n$ and using the discrete Gronwall's inequality, we obtain for all $\ell=1,\dots,N$
\begin{align*}
\norm{\eT^\ell}{\Ldrc}^2&
+\Dt\sum_{n=1}^\ell\norm{\eT^n}{\Huc}^2
\lesssim
\norm{\eT^0}{\Ldrc}^2
+\underset{0\leq n\leq N-1}{\max}\norm{A^n-I_h A^n}{\Htu}^2
\nonumber\\
&+\underset{0\leq n\leq N}{\max}\norm{Q_hT^n-T^n}{\Huc}^2
+\int_{0}^{\Tf}\norm{(Q_h-I)\partial_t T(s)}{\Ldrc}^2\ds \\
&+\Dt^2\int_0^{\Tf}
\left\{
\norm{\partial_t T_c(s)}{\LdGur}^2
+\norm{\partial_{tt} T(s)}{\Ldrc}^2
+\norm{\partial_t A(s)}{\Huc}^2
+\norm{\partial_t T(s)}{\Ldrc}^2
\right\}\ds.
\end{align*}
Notice that $\norm{\eT^0}{\Ldrc}=0$ since we have assumed $T_h^0:=Q_hT^0$.
Thus, the result follows by writing $A_h^n-A^n=\eA^n+I_h A^n-A^n$, $T_h^n-T^n=\eT^n+Q_h T^n-T^n$, from the previous inequality and \eqref{etanh}.
\end{proof}

Next, as a consequence of the previous theorem and  interpolation properties for $I_h$ and
$Q_h$ (cf.~\eqref{est_Ih}--\eqref{est_Qh}), we obtain the following result, which yields the optimal convergence order of the proposed scheme under additional spatial regularity assumptions for both the temperature and the magnetic vector potential.

\begin{theorem}\label{thm:est_2}
Under the hypothesis of the previous theorem, if we further assume $A$ in $\Ccal([0,\Tf];\H^{2}_1(\O)\cap \HtuG)$ and
$T$ in $\Ccal([0,\Tf];\H^{2}_1(\Ocero))\cap\H^1(0,\Tf;\H^{1}_1(\Ocero))$, then
\begin{multline*}
\underset{1\leq n\leq N}{\max}\norm{T_h^n-T^n}{\Ldrc}^2
+\Dt\sum_{n=1}^N\norm{T_h^n-T^n}{\Huc}^2
+\underset{0\leq n\leq N}{\max}\norm{A_h^n-A^n}{\Htu}^2\\
\lesssim
h^2\Big(\norm{A}{\Ccal([0,\Tf];\H^{2}_1(\O)\cap \Htu)}^2
+\norm{T}{\Ccal([0,\Tf];\H^{2}_1(\Ocero))}^2
+\norm{\partial_t T}{\L^2(0,\Tf;\H^{1}_1(\Ocero))}^2
\Big)
\\
+\Dt^2\Big(\norm{\partial_t T_c}{\L^2(0,\Tf;\LdGur)}^2
+\norm{ T}{\H^2(0,\Tf;\Ldrc)}^2
+\norm{\partial_t A}{\L^2(0,\Tf;\Huc)}^2
\Big).
\end{multline*}
\end{theorem}

The solution $(A^n_h, T^n_h)$ of \cref{pbm:AT_fully} allows us to compute
an approximation to the magnetic induction $\bB$  defined by
\[
\bB^n_h\defeq
-\frac{\partial A^n_h}{\partial z}\be_r
+\frac{1}{r}\frac{\partial(r A^n_h)}{\partial r}\be_z, \quad n=0,\dots,N.
\]

\begin{corollary} Under the same assumptions as in Theorem~\ref{thm:est_2} it follows that
\begin{multline*}
\underset{1\leq n\leq N}{\max}
\norm{\bB^n_h-\bB^n}{\L^2(\Ot)^3}^2
\lesssim
h^2\Big(\norm{A}{\Ccal([0,\Tf];\H^{2}_1(\O)\cap \Htu)}^2
+\norm{T}{\Ccal([0,\Tf];\H^{2}_1(\Ocero))}^2
+\norm{\partial_t T}{\L^2(0,\Tf;\H^{1}_1(\Ocero))}^2
\Big)
\\
+\Dt^2\left(
\norm{\partial_t T_c}{\L^2(0,\Tf;\LdGur)}^2
+\norm{T}{\H^2(0,\Tf;\Ldrc)}^2
+\norm{\partial_t A}{\L^2(0,\Tf;\Huc)}^2
\right).
\end{multline*}
where $\bB^n\defeq
-\frac{\partial A^n}{\partial z}\be_r
+\frac{1}{r}\frac{\partial(r A^n)}{\partial r}\be_z $, $n=0,\dots,N$.
\end{corollary}

%-------------------------------------------------------------------------------------
\section{Numerical tests}
\label{sec:num}
In this section, we  present some numerical results obtained by solving \cref{pbm:AT_fully} described in the previous section. First, we conduct a test to assess the error estimates.
%In the absence of an analytical solution, we evaluate the orders of convergence
%by comparing the results with those obtained by the same method on highly refined meshes.
% Subsequently, we employ the proposed scheme to address an industrial challenge entailing
% the numerical simulation of an induction hardening process.
We then apply the proposed scheme to the numerical simulation of an industrial induction-hardening process.

\subsection{Convergence analysis}
For this test, we consider a conducting cylinder surrounded by a copper induction coil modeled by four individual rings. The geometric configuration is the same as that used in Test 6.1 of~\cite{GLSV1}, but all coordinates are multiplied by
100 in order to obtain a setting where both spatial and temporal errors can be observed.
% but with all coordinates multiplied by 100 in order to create an example in which the spatial and temporal variations allow both types of errors to be captured.
The electrical frequency is $50$ Hz, as in Test 6.1 of~\cite{GLSV1}, while the voltage sources are multiplied by a factor of $10^4$ to simulate a short time interval. The final time is $0.5$ seconds, and the initial temperature is $20\ ^\circ\mathrm{C}$.

Regarding the thermal properties, the density, specific heat and thermal conductivity are taken as $\rho = 7500\, \mathrm{kg\,m^{-3}}$,  $c_P = 510\, \mathrm{J\,K^{-1}\,kg^{-1}}$, and $\kappa=15\, \mathrm{W\,m^{-1}\,K^{-1}}$, respectively. On the  thermal boundary, we use a heat transfer coefficient a heat transfer coefficient $\eta = 5\ \mathrm{W\,m^{-2}\,K^{-1}}$ and a convection temperature  $T_c = 20\ ^\circ\text{C}$. For the electrical properties, the magnetic permeability is taken equal to the vacuum permeability for all materials, $\mu_0 = 4\pi \times 10^{-7}\ \mathrm{H\,m^{-1}}$. The electrical conductivity in the coil is constant and equal to $2\times10^{7}\,\mathrm{S\,m^{-1}}$, while in the workpiece it depends on temperature as in~\cite{GLSV1}, but multiplied by $10^4$ so that the skin effect is less pronounced:
$$\sigma(T) = 1000 \times \left(-4.51\times10^{-13}T^2 + 8.18\times10^{-10}T + 7.26\times10^{-7}\right)\ \text{S/m}.$$
We focus on the errors for the temperature field. In the absence of an analytical solution, the reference solution $T^{\text{ref}}$
was obtained using the same finite-element method on a highly refined mesh and with a very small time-step.
The numerical solution of \cref{pbm:AT_fully} is then computed on a sequence of progressively refined meshes derived from a coarse one, while the time step is successively reduced from $\Delta t = 0.25$ s.

Figure~\ref{fig:domain_errors} shows log--log plots of the error in $T$ measured in the discrete norm $\L^2(0,\Tf;\Huc)$ versus the number of degrees of freedom (d.o.f.) and versus $\Delta t$ (left and right, respectively). To study the error with respect to the number of degrees of freedom, we set the time step to a sufficiently small value so that the error depends almost exclusively on the mesh size $h$. In this case, a linear dependence on $h$ is observed. Similarly, to study the error with respect to the time step, the mesh is chosen sufficiently fine so that the time discretization error dominates. A linear dependence on  $\Delta t$ is again observed.
%{\bf Hablé con Pablo y cree que es mejor dejar solo las graficas porque se repite información de la tabla. El proponía la tabla por si ocupaba menos.Ahora mismo dejé las dos cosas, pero el cree que solo una mejor }

%-----------------------
%\begin{table}[htbp]
%\centering
%\caption{%
%Convergence for $\|T_h - %T^{\text{ref}}\|_{L^2(0,\mathcal{T};\,H^1(\Omega_0))}$.
%}
%\label{tab:convergence_time2}
%\renewcommand{\arraystretch}{1.3}
%\setlength{\tabcolsep}{4pt}
%\begin{tabular}{r r r r}
%\toprule
%$\Delta t$ & $N_{\mathrm{d.o.f.}}$ & $\|T_h - T\|$ & rate \\
%\midrule
%\multirow{4}{*}{$4.88\times 10^{-4}$}
%& $788$   & $6,67\times10^{-1}$ & ---  \\
%& $3087$  & $3,21\times10^{-1}$ & $1.07$ \\
%& $6898$  & $2,15\times10^{-1}$ & $1.08$ \\
%& $12221$ & $1,67\times10^{-1}$ & $0.88$ \\

%\bottomrule
%\end{tabular}
%\begin{tabular}{r r r r}
%\toprule
%$N_{\mathrm{d.o.f.}}$ & $\Delta t$ & $\|T_h - T\|$ & rate \\
%\midrule
%\multirow{4}{*}{$109237$}
%& $0.25$ & $1.17$ & ---  \\
%& $0.125$ & $5.81$ & $1.00$ \\
%& $0.0625$ & $2.88$ & $1.01$ \\
%& $0.03125$ & $1.42$ & $1.02$ \\
%\bottomrule
%\end{tabular}
%\end{table}

\begin{figure}[!h]
\centering
\includegraphics*[width=0.45\textwidth]{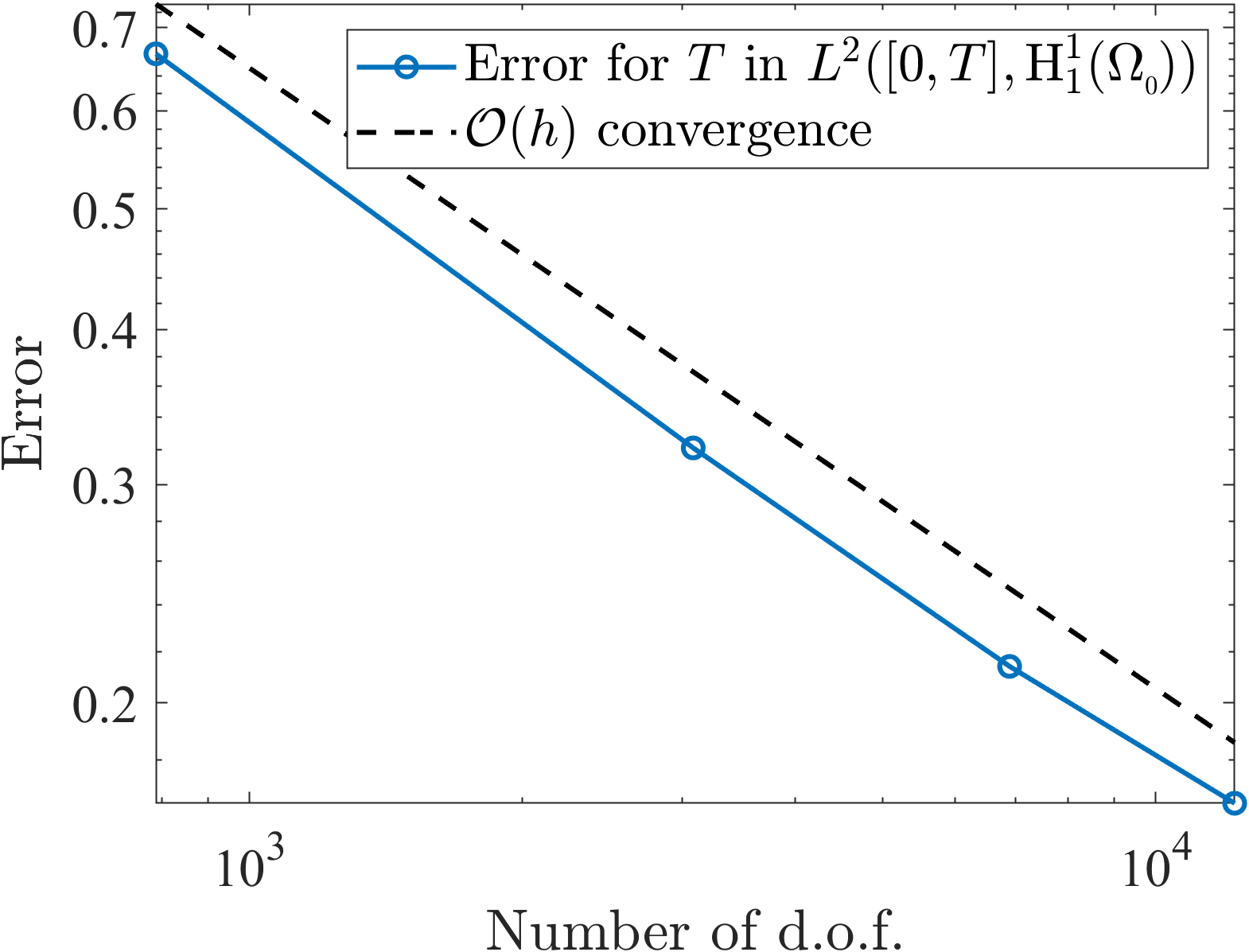}
\hspace{5mm}
\includegraphics*[width=0.46\textwidth]{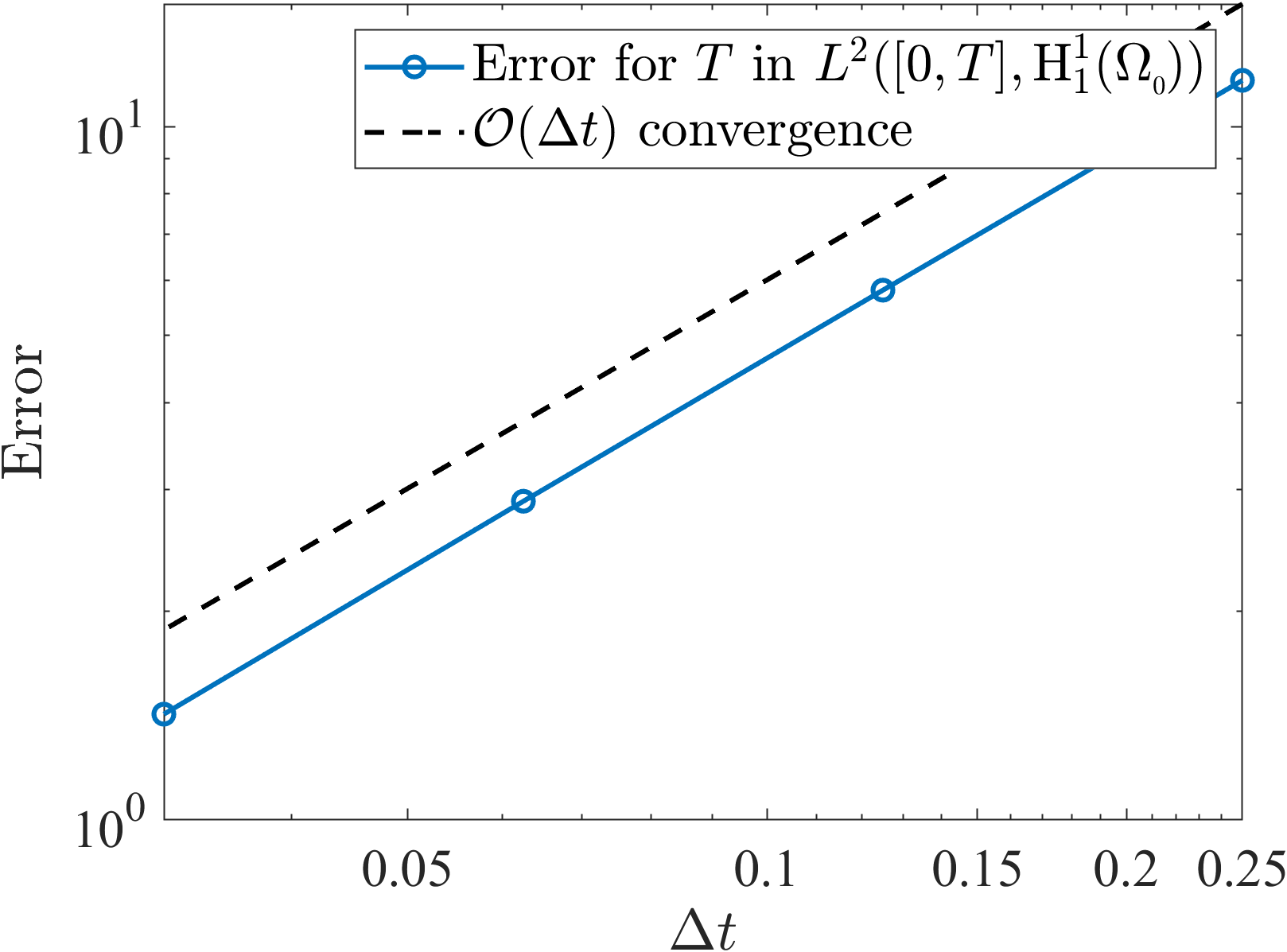}
\caption{Errors in log-log scale for $T$ versus the number of degrees of freedom (left) and versus $\Delta t$ (right).}
\label{fig:domain_errors}
\end{figure}
%-------------------------------------------------------------------------------------
\subsection{An industrial application of the induction heating model}
%\caqm{Induction hardening process: This specific application utilizes the rapid and localized heating provided by induction heating to alter the microstructure of metals, enhancing their surface hardness and mechanical properties.
% \caqm{In this section we have performed the numerical simulation of an induction hardening process typical
% in the automotive industry. Induction hardening is a surface treatment process that uses the rapid and localized heating provided by induction heating to change the microstructure of metals, improving their surface hardness and mechanical properties. In this process, the workpiece is rapidly heated by induced eddy currents in the surface layer, followed by controlled quenching. This rapid thermal cycle results in the formation of martensite, a hard and wear-resistant phase, in the surface region of the material. By fine-tuning the frequency, power and duration of the heating cycle, induction hardening can be tailored to achieve specific hardness profiles and depths. The frequency of the AC current plays a crucial role in determining the depth of penetration of the eddy currents into the workpiece. This phenomenon, known as the skin effect, causes the eddy currents to concentrate near the surface of the conductor. The higher the frequency of the AC current, the lower the depth of penetration of the induced currents.
% This technique is widely used in industries such as automotive and aerospace for components requiring increased surface durability, such as shafts, gears, axles and bearings.
% }
In the automotive industry, induction hardening is widely used as a surface treatment for carbon-steel components. The process consists of rapidly and locally heating a thin surface layer by electromagnetic induction in order to modify the microstructure and, after  cooling, improve hardness and wear resistance. Here, we focus on the induction heating stage and present numerical simulations for a typical axle stub. The analysis concerns the temperature field generated by eddy currents in the near-surface region, which depends mainly on the inductor geometry, the excitation frequency, the coil current, and the heating time. In particular, the AC frequency determines the penetration depth of the induced currents through the skin effect. In this example, the operating frequency is $f \!=  \!6000$ Hz and the coil current is $I \!= \! 4000$ A rms; a representative coil voltage is 33.1 V rms with phase 0.46 rad. The inductor is made of copper and the workpiece of ferromagnetic steel. %(see~\cite{Batti2009}).
Its magnetic behavior is described by a temperature-dependent Fröhlich-Kennelly law (see \cite{Petzold2014})
%\vspace{-2mm}
\begin{equation}
\bB=\mu_0\,\bH\;+\; f(\hat{T})\,\frac{\bH}{a+b|\bH|},\quad f(\hat{T})=\left(\frac{T_c^2-\hat{T}^2}{T_c^2-T_0^2}\right)^{\!1/4},
\label{eq:BH_FK}
\end{equation}
with parameters $a,b>0$. From \eqref{HnuB} and \eqref{eq:BH_FK} we obtain
\[
\nu(\hat T,|\bB|)=\frac{-a\mu_0-f(\hat T)+b|\bB|+\sqrt{D}}{2b\mu_0|\bB|},
\quad
D=(a\mu_0+f(\hat T)-b|\bB|)^2+4ab\mu_0|\bB|.
\]
The parameters are chosen as $a=1128.27$, $b=0.64$, $T_c=750^\circ\mathrm{C}$ (Curie temperature),  $T_0=20^\circ\mathrm{C}$, and  $\hat{T}=400^\circ\mathrm{C}$, with $\mu_0=4\pi\times 10^{-7}\,\mathrm{H\,m^{-1}}$. Figure~\ref{fig:steel_properties} shows the resulting B-H curve together with the conductivity of the steel as a function of temperature. The thermal conductivity  and volumetric heat capacity are taken as $\kappa=45\, \mathrm{W\,m^{-1}\,K^{-1}}$, $\rho\, c_P= 3925000\, \mathrm{J\,m^{-3}\,K^{-1}}$.

Although the workpiece exhibits a nonlinear B-H relation, the electromagnetic problem is treated under the time-harmonic assumption. This is justified by the different time scales of the electromagnetic and thermal phenomena, which makes such an approximation appropriate in the present setting.

\begin{figure}[!h]
\centering
\includegraphics*[scale=0.43]{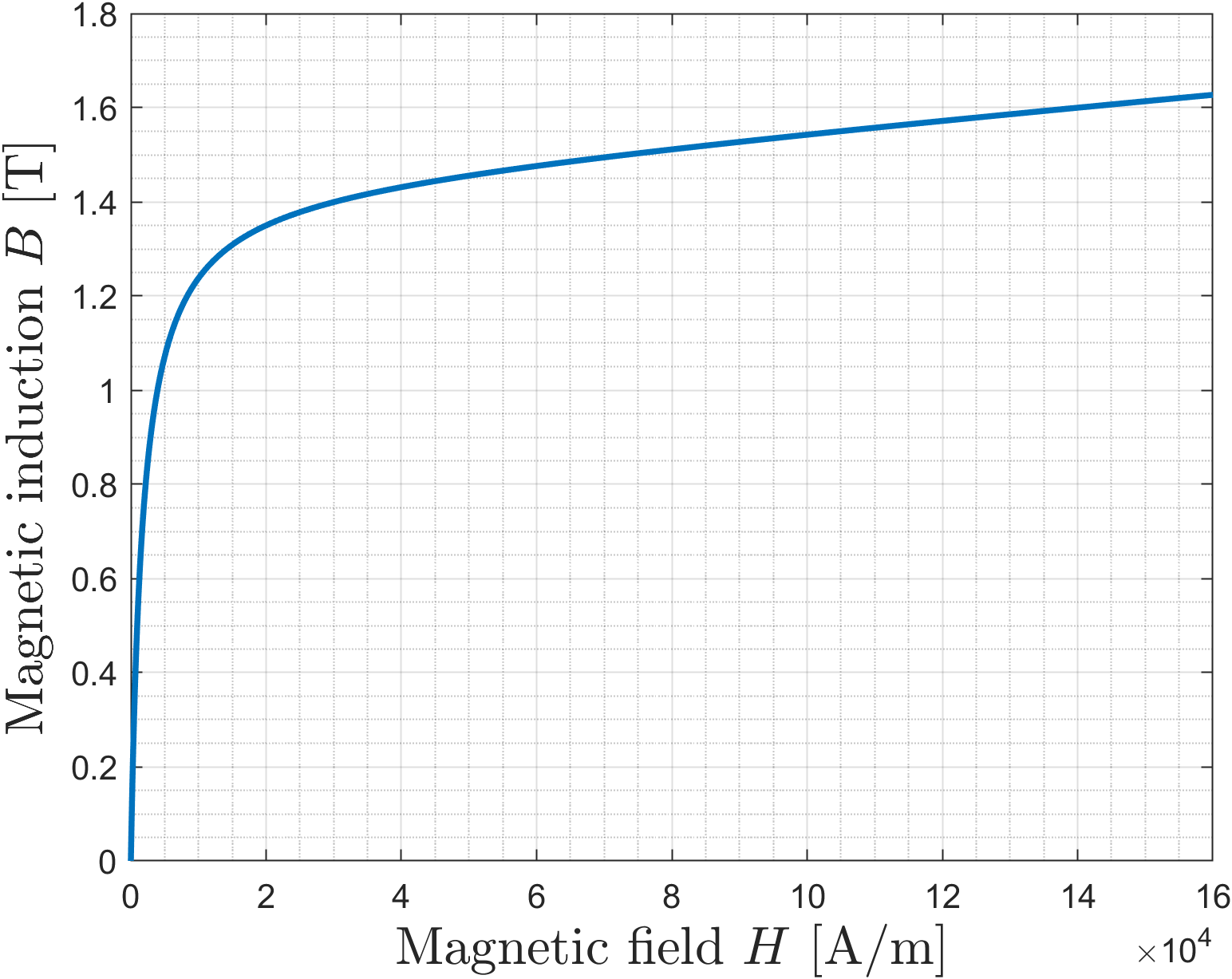} \hspace{5mm}
\includegraphics*[scale=0.43]{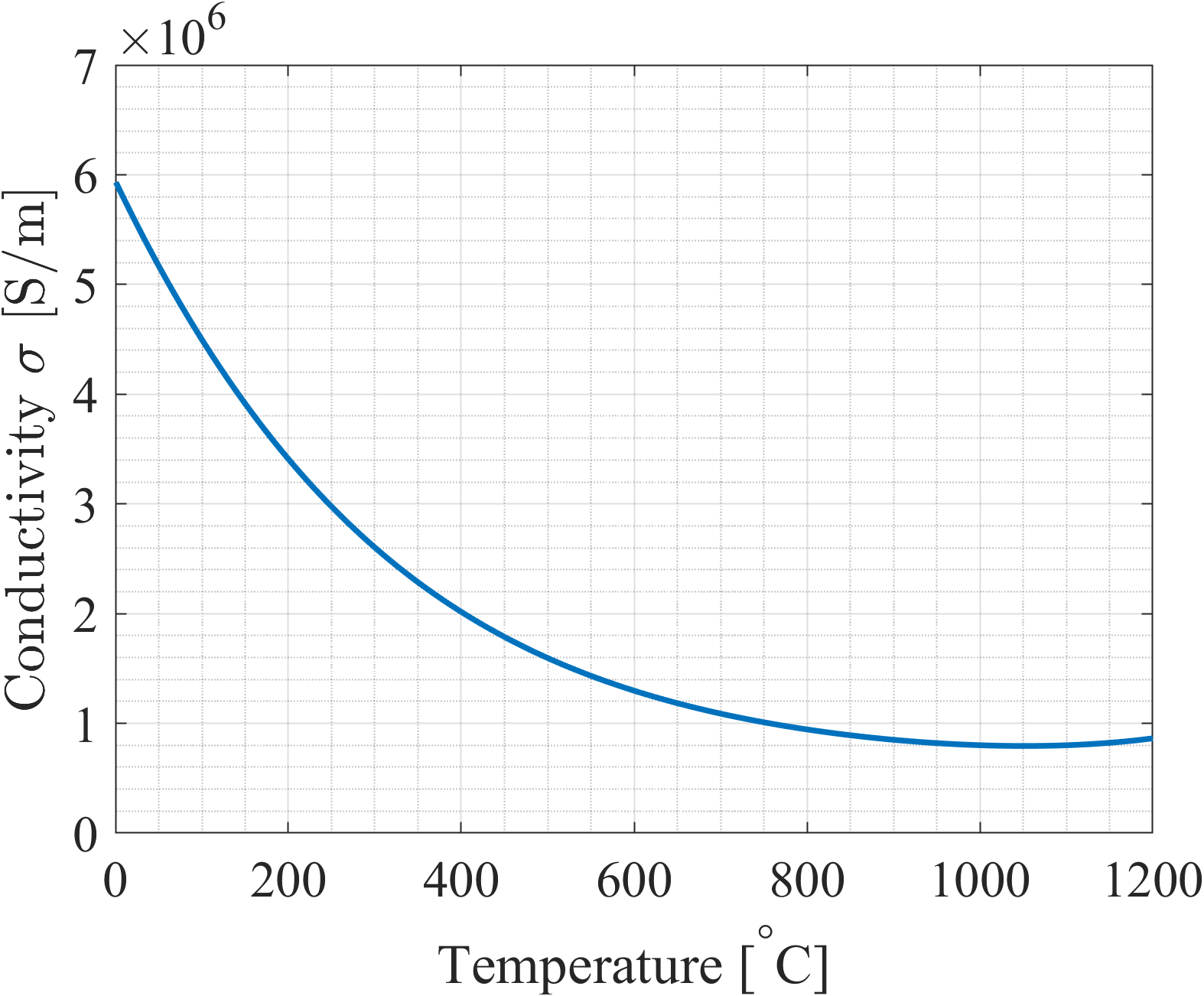}
\caption{Steel properties: B-H curve (left) and $\sigma(T)$  (right).}
\label{fig:steel_properties}
\end{figure}

Figure~\ref{fig:temp_evolution} displays the temperature evolution in the workpiece at several time instants in [0,2] s. The results show a progressive increase of the temperature, essentially localized near the surface region closest to the coil,  while the interior remains comparatively cold, in agreement with the expected behavior.
\begin{figure}[!h]
\centering
 \includegraphics*[scale=0.45]{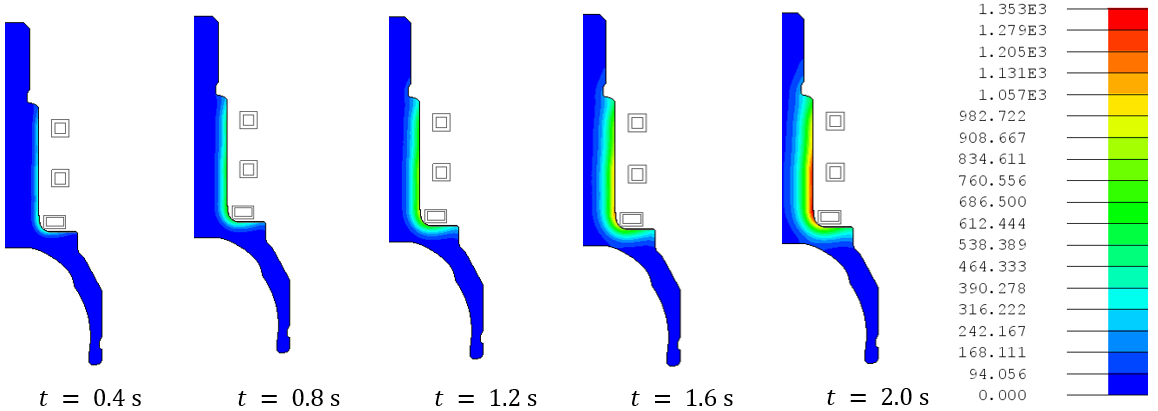}
\caption{Evolution of the temperature ($^\circ\mathrm{C}$) in the workpiece.}
% \cmag{Figura superior: con curva BH dependiendo de la temperatura y radiación. Figura inferior: con curva BH de la Fig. 3 y sin radiación}
\label{fig:temp_evolution}
\end{figure}

%-------------------------------------------------------------------------------------
\section*{Acknowledgments}
%-------------------------------------------------------------------------------------
This work was supported by FEDER, Ministerio de Econom\'ia, Industria y Competitividad-AEI
research project PID2021-122625OB-I00, by Xunta de Galicia (Spain) research project ED431C 2025/09. P.~Venegas  was partially supported by  ANID-Chile through FONDECYT grant 1211030 and Centro de Modelamiento
Matemático (CMM), grant FB210005, BASAL funds for centers of excellence.
B.~L\'opez-Rodr\'iguez was partially supported by Universidad Nacional de Colombia through
Hermes project $63726$.
%-------------------------------------------------------------------------------------

\bibliographystyle{plain}
\bibliography{GLSV2_ref_NO_URL}
\end{document}